\documentclass[english, 11pt]{amsart}
\usepackage{amssymb,bm,amsmath, mathrsfs, amsxtra,amsthm,amscd }
\usepackage[all]{xy}
\usepackage{color}

\makeatletter
\def\@tocline#1#2#3#4#5#6#7{\relax
  \ifnum #1>\c@tocdepth 
  \else
    \par \addpenalty\@secpenalty\addvspace{#2}%
    \begingroup \hyphenpenalty\@M
    \@ifempty{#4}{%
      \@tempdima\csname r@tocindent\number#1\endcsname\relax
    }{%
      \@tempdima#4\relax
    }%
    \parindent\z@ \leftskip#3\relax \advance\leftskip\@tempdima\relax
    \rightskip\@pnumwidth plus4em \parfillskip-\@pnumwidth
    #5\leavevmode\hskip-\@tempdima
      \ifcase #1
       \or\or \hskip 1em \or \hskip 2em \else \hskip 3em \fi%
      #6\nobreak\relax
    \hfill
    \hbox to\@pnumwidth{\@tocpagenum{#7}}\par
    \nobreak
    \endgroup
  \fi}
\makeatother
  
\usepackage{hyperref}

\usepackage[OT2,T1]{fontenc}
\DeclareSymbolFont{cyrletters}{OT2}{wncyr}{m}{n}
\DeclareMathSymbol{\Sha}{\mathalpha}{cyrletters}{"58}

\theoremstyle{plain}

\swapnumbers
\newtheorem{theorem}[subsubsection]{Theorem}
\newtheorem{proposition}[subsubsection]{Proposition}
\newtheorem{lemma}[subsubsection]{Lemma}
\newtheorem{corollary}[subsubsection]{Corollary}
\newtheorem*{claim*}{Claim}
\newtheorem{conjecture}[subsubsection]{Conjecture}

\theoremstyle{definition}
\newtheorem{definition}[subsubsection]{Definition}

\theoremstyle{remark}
\newtheorem{remark}[subsubsection]{Remark}

\theoremstyle{example}
\newtheorem{example}[subsubsection]{Example}

\numberwithin{equation}{subsubsection}

\bigskip

\def\frakp{{\mathfrak p}}
\def\frakh{{\mathfrak h}}

\def\fraku{{\mathfrak u}}
\def\frakg{{\mathfrak g}}
\def\frakm{{\mathfrak m}}

\def\frakb{{\mathfrak b}}
\def\fraka{{\mathfrak a}}

\def\id{\operatorname{id}} 
\def\rad{\operatorname{rad}}

\def\Spec{\operatorname{Spec}}

\def\diag{\operatorname{diag}}  
\def\Int{\operatorname{Int}}
\def\Ad{\operatorname{Ad}} 
 
\def\Lie{\operatorname{Lie}}

\def\im{\operatorname{im}} 
\def\ker{\operatorname{ker}} 
\def\pr{\operatorname{pr}}
\def\GL{\operatorname{GL}}

\def\RZ{\mathrm{RZ}}

\newcommand{\iso}{\xrightarrow{\sim}}

\newcommand{\q}{{\mathbb Q}}
\newcommand{\Z}{\mathbb Z}
\newcommand{\z}{\mathbb Z}
\newcommand{\R}{\mathbb R}
\newcommand{\C}{\mathbb C}

\newcommand{\co}{\mathcal O}

\newcommand{\calRZ}{{\mathcal {RZ}}}

\def\B{\mathcal B}\def\C{{\mathcal C}}\def\D{\mathcal D}
\def\O{\mathcal O}\def\P{\mathcal P}

\newcommand{\Gm}{{\mathbb{G}_m}}
\newcommand{\g}{{\mathbb{G}}}

\newcommand{\Db}{{\mathbb D}}

\def\a{\mathbb A}
\def\f{\mathbb F}\def\g{\mathbb G}
\def\q{\mathbb Q}\def\z{\mathbb Z}

\newcommand\ANilp{{\mathrm {ANilp}}}

\newcommand{\quash}[1]{}  

\begin{document}
\author[O. B\"ultel]{O. B\"ultel}
 \address{O. B\"ultel: Aldegreverstrasse 28\\ 45147 Essen\\ Germany }
\author[G. Pappas]{G. Pappas}
 \thanks{G.P. is partially supported by NSF grants   DMS-1360733 and DMS-1701619.}

\address{G. Pappas: Dept.~of Mathematics\\ Michigan State University\\ E. Lansing\\ MI 48824\\ USA}

\title[$(G,\mu)$-displays and Rapoport-Zink spaces]{$(G,\mu)$-displays and Rapoport-Zink spaces}
\date{\today}

\begin{abstract}

Let $(G, \mu)$  be a pair of a reductive group $G$ over the $p$-adic integers and a minuscule cocharacter $\mu$ of $G$ defined over an unramified extension.
We  introduce and study ``$(G, \mu)$-displays'' which generalize Zink's Witt vector displays. We use these to define certain Rapoport-Zink formal schemes purely group theoretically, i.e. without  $p$-divisible groups.   

\end{abstract}

\maketitle

\tableofcontents

\vfill\eject

\section{Introduction}
In the theory of Shimura varieties, as interpreted by Deligne \cite{DeligneCorvallis}, one starts with a ``Shimura datum''
$(G, X)$.
This is a pair  
of a (connected) reductive algebraic group $G$ over the field of rational numbers $\q$, and a symmetric Hermitian domain $X=\{h\}$
for $G(\mathbb R)$ given as a $G(\mathbb R)$-conjugacy class of an algebraic group homomorphism
$h: {\mathbb C}^*\to G_{\R}$ over the real numbers $\R$, that satisfies certain axioms.  For each open compact subgroup $K$ of the finite adeles $G({\mathbb A}^f)$, one considers the Shimura variety
${\rm Sh}_K(G, X)=G(\q)\backslash (X\times G({\mathbb A}^f)/K)$; this complex analytic space 
is actually an algebraic variety with a canonical model over a number field, the so-called reflex field of
the pair $(G, X)$. 

The proposal that there should exist a similar theory of ``$p$-adic local Shimura varieties'' was recently put forward by Rapoport and Viehmann
\cite{RapoportVi}. The current paper can be viewed as a contribution to this  theory. In \cite{RapoportVi}, one starts with the ``local Shimura datum''. This is  
a triple $(G,  \{\mu\}, [b])$ consisting of a connected reductive algebraic group $G$ over 
$\q_p$,  a conjugacy class 
$\{\mu\}$ of minuscule cocharacters of $G_{\overline\q_p}$, and a $\sigma$-conjugacy class $[b]$ of elements in $G(L)$, satisfying some simple axioms. (Here, $L$ is the completion 
of the maximal unramified extension of $\q_p$  and $\sigma$ the canonical lift of Frobenius).
For each 
open compact subgroup $K\subset G(\q_p)$ we should have the local Shimura variety ${\rm M}_K(G,  \{\mu\} , [b])$;
this is expected to be a rigid analytic space   with a canonical model over the ``local reflex field'' which is a finite extension $E$ of $\q_p$ that depends only on $(G, \{\mu\})$. See \cite{RapoportVi} for more details and expected properties of the local Shimura varieties. 
Examples of such local Shimura varieties have first been constructed  in the work of Rapoport and Zink \cite{RapZinkBook}
in some special cases.
There, they appear as covers of the generic fibers of certain formal schemes over the ring of integers $\co_E$. These formal schemes 
(which we call Rapoport-Zink formal schemes) are moduli spaces parametrizing 
 $p$-divisible groups (with additional structure) with a quasi-isogeny to a fixed $p$-divisible 
 group. They can be viewed as integral models of the desired local Shimura varieties.
 
In this paper, we consider the case in which  the local Shimura datum is {\sl unramified}.
In particular, the open compact subgroup $K$
is maximal hyperspecial. Then, starting from $(G,  \{\mu\}, [b])$, we give a functor 
${\mathrm {RZ}}_{G, \mu, b}$ on $p$-nilpotent algebras.
As we will explain below, this functor has a direct group theoretic definition which uses only $G$, suitable representatives $\mu$ and $b$,
and involves rings of Witt vectors.
 We conjecture that ${\mathrm {RZ}}_{G, \mu, b}$ is represented by a formal scheme which should then be
an integral model of the sought-after local Shimura variety for this hyperspecial level. We essentially show this conjecture when 
the local Shimura datum is of {\sl Hodge type}, i.e. when it  embedds  into a local Shimura datum for ${\rm GL}_n$.

The main tool we use is a variation of the theory of Zink displays. By the work of Zink and Lau, (formal) $p$-divisible groups over a $p$-adically complete and separated algebra $R$ are classified by displays. These are projective finitely generated modules $P$ over the ring of Witt vectors $W(R)$ with additional structure given by a suitable filtration $I(R)P\subset Q\subset P$ and a Frobenius semi-linear operator $V^{-1}: Q\to P$ that satisfy certain axioms. Here, $I(R)$ is the kernel of the  projection $w_0: W(R)\to R$ given by $w_0(r_0, r_1, \ldots )=r_0$.

In this paper, we develop a theory of  displays ``with $(G, \mu)$-structure'': Instead of projective $W(R)$-modules, we use $G$-torsors over $W(R)$ or, equivalently, $L^+G$-torsors over $R$. 
Here, $G$ is a reductive group scheme over $\Z_p$ and $\mu$ is a (minuscule) cocharacter
of $G$ defined over a finite unramified extension $W(k_0)$.
Then $L^+G$ is the positive ``Witt loop group scheme of $G$'' defined by $L^+G(R)=G(W(R))$. To explain the definition of $(G,\mu)$-displays we need
to introduce some more objects: We let $H^\mu$  be the subgroup scheme of $L^+G_{W(k_0)}$ with $R$-valued points  $H^\mu(R)$ given 
by those $g\in G(W(R))$ whose projection $g_0\in G(R)$ lands in the $R$-points of the parabolic subgroup $P_\mu\subset G$ associated to $\mu$. We then construct the ``divided Frobenius'' which 
is a group scheme homomorphism
\[
\Phi_{G, \mu}: H^\mu \to L^+G_{W(k_0)}
\]
such that, for $h\in H^\mu(R)$,   we have
\[
\Phi_{G, \mu}(h)=\mu^\sigma(p)\cdot F(h)\cdot \mu^{\sigma}(p)^{-1}
 \]
 in $G(W(R)[1/p])$. Here $F$ is induced by the Frobenius on $W(R)$.

\begin{definition} A $(G, \mu)$-display is a triple $\D:=(P, Q, u)$ which we can write
\[
P\hookleftarrow Q \xrightarrow{u} P,
\]
where $P$ is a $L^+G$-torsor and $Q$ a $H^\mu$-torsor. We ask that  \[
P=Q\times_{H^\mu}L^+G_{W(k_0)}\] and that $u$ is a morphism compatible with the actions
on $Q$ and $P$ and with $\Phi_{G, \mu}$, in the sense that $u(q\cdot h)=u(q)\cdot \Phi_{G, \mu}(h)$, for all $q\in Q(R)$, $h\in H^\mu(R)$.
\end{definition} 

(This definition  first appeared in \cite{E7}, see also \cite{ZhangCao} for a similar construction. One can view this structure as formally similar to that of a Drinfeld shtuka with $W(R)$ replacing the affine ring $R[t]$ of a curve.)

Locally the morphism $u$ is given by
a point of $L^+G$ and we can see that $(G, \mu)$-displays are objects of the quotient stack
\[
[L^+G_{W(k_0)}/_{\Phi_{G, \mu}} H^\mu]
\]
where the action of $H^\mu$ is by $\Phi_{G, \mu}$-conjugation: $g\cdot h:= h^{-1}g\Phi_{G, \mu}(h)$. 

Apparently, the notion of a $(G,\mu)$-display is sufficiently well-behaved and we can generalize several of the results of Zink on Witt vector displays, for example, about deformation theory.
We can also define a notion of a $G$-quasi-isogeny between two $(G, \mu)$-displays and show that a triple 
$(G, \mu, b)$ allows us to give a ``base-point'' $(G, \mu)$-display $\D_0$ defined over $k=\bar k_0$. 
With these ingredients, we now describe the ``Rapoport-Zink functor'' ${\mathrm {RZ}}_{G, \mu, b}$: 

By definition, ${\mathrm {RZ}}_{G, \mu, b}$  sends a $p$-nilpotent  $W(k)$-algebra $R$ to the set of isomorphism classes of pairs
$(\D, \rho)$, where $\D$ is a $(G,\mu)$-display over $R$, and $\rho:  \D\times_{R}R/pR \dashrightarrow  \D_0\times_k R/pR$ is 
a $G$-quasi-isogeny.

  We can also view ${\mathrm {RZ}}_{G, \mu, b}$ as given by a quotient stack. Set $LG$ for the Witt loop group scheme of $G$ given by $LG(R)=G(W(R)[1/p])$. Then
$\RZ_{G,\mu, b}$ is given by the isomorphism classes of objects of the
  (fpqc, or \'etale) quotient stack 
  \[
{\mathcal {RZ}}_{G, \mu, b}:= [ (L^+G\times_{LG,\mu, b} LG)/ H^\mu]
  \]
for the action of $H^\mu$ given by
\[
(U, g)\cdot h=(h^{-1}\cdot U\cdot  \Phi_{G,\mu}(h), g\cdot h).
\]
Here, the fiber product $(L^+G\times_{LG,\mu, b} LG)(R)$ is  by definition the set of pairs $(U, g)$
with $U\in L^+G(R)$, $g\in LG(R)$, such that 
\[
g^{-1}bF(g)=U\mu^\sigma(p)
\]
 in $LG(R)$. The $\sigma$-centralizer group 
 \[
 J_b(\q_p)=\{j\in G(L)\ |\ j^{-1}b\sigma(j)=b\}
 \]
 acts on ${\mathcal {RZ}}_{G, \mu, b}$ by 
 \[
 j\cdot (U,g)=(U, j\cdot g).
 \]
 
It follows from the definition that the $k$-valued points of ${\mathrm {RZ}}_{G, \mu, b}$
are given by the affine Deligne-Lusztig set
\[
{\mathrm {RZ}}_{G, \mu, b}(k)= \{g\in G(L)\, |\, g^{-1}b\sigma(g)\in G(W)\mu^\sigma(p)G(W)\}/G(W).
\]
Here, $W=W(k)$, $L=W(k)[1/p]$.

Assuming an additional mild condition on the slopes of $b$, we conjecture that the functor ${\mathrm {RZ}}_{G, b, \mu}$ is representable by a formal $W$-scheme
which is formally locally of finite type and $W$-formally smooth. For $G=\GL_n$, and when 
$b$ has no zero slopes, this follows from the results of Rapoport-Zink and Zink and Lau.
Indeed, in that case, the results of Zink and Lau imply that our functor is equivalent to the functor of isomorphism classes of 
(formal) $p$-divisible groups with a quasi-isogeny to a fixed $p$-divisible group
considered in \cite{RapZinkBook}; the representability of that functor is one of the main 
results of loc. cit. For more general $(G,\mu)$, we show that pairs $(\D, \rho)$ over $R$
have no automorphisms when $R$ is Noetherian. 

When $(G,\mu)$ is of Hodge type,
i.e. when there is an embedding $i: G\hookrightarrow \GL_n$ with $i\cdot\mu$ conjugate to one of the standard minuscule cocharacters of $\GL_n$, we show that the restriction of the functor 
${\mathrm {RZ}}_{G, \mu, b}$ to Noetherian algebras is representable as desired. 
This is one of the main results of the paper. 
The basic idea of the proof  is as follows: We  show (Corollary \ref{closedRZ}) that when $(G, \mu)\hookrightarrow (G',\mu')$ is an embedding of local Shimura data, the corresponding 
morphism of  stacks ${\mathcal {RZ}}_{G, \mu, b}\to  {\mathcal {RZ}}_{G',  \mu', i(b)} $,  when restricted to Noetherian algebras, is relatively representable by a closed immersion. (Proving this is, maybe surprisingly, involved; the main step is the descent result of Proposition \ref{analog}. Part of the difficulty comes from the need to handle the nilradical of various rings that appear in the argument.) 
In the Hodge type case, we have an embedding 
$(G, \mu)\hookrightarrow (\GL_n,\mu')$ and ${\mathcal {RZ}}_{\GL_n, \mu', i(b)} $
is representable by the results of Zink, Lau, and Rapoport-Zink as explained above.
The result follows by combining these two statements. (In all of this, we have to assume 
that the slopes of $i(b)$ do not include $0$.) 

A construction of Rapoport-Zink formal schemes in the Hodge type case was also given by Kim \cite{KImRZ},
and, under an additional condition,  independently by Howard and the second author \cite{HP2}.
It is easy to see that,
in the Hodge type case,   the formal scheme representing ${\mathrm {RZ}}_{G, \mu, b}$ given in this paper
coincides with the corresponding formal schemes constructed in \cite{RapZinkBook}
(in the PEL and EL cases), in \cite{KImRZ}, and in \cite{HP2}. Hence, our results
 give a unified group theoretic description of these formal schemes as moduli functors
 and describe their $R$-valued points for all Noetherian algebras $R$. 
 In particular, they imply the existence of isomorphisms between ``classical'' Rapoport-Zink spaces
 when the corresponding local Shimura data are isomorphic (for example, because 
 of exceptional isomorphisms between the underlying groups); this answers a question of Rapoport.  
 
 Let us also mention here that the restriction of our functors to perfect $k$-algebras
 has already appeared in the work of Zhu \cite{ZhuWitt} (see also  \cite{BhattScholzeWitt}). In fact, when we consider functors with values 
 in perfect algebras there is a more comprehensive theory that employs the ``Witt vector affine Grassmannian''
 which does not require the assumption that $\mu$ is minuscule. However, the techniques of \cite{ZhuWitt} and \cite{BhattScholzeWitt}
 cannot handle $p$-nilpotent algebras and only give information about the perfection of the special fiber.
 On the other extreme, when one considers only the generic fibers, in the Hodge type case, 
 Scholze and Scholze-Weinstein \cite{ScholzeWeinsteinModuli} can give a construction 
 of the inverse limit of the tower of local Shimura varieties 
 as a perfectoid space (see also \cite{CaraianiScholze}). There are also related constructions
 of more general spaces (even for $\mu$ not minuscule) that use  $G$-bundles 
 on the Fontaine-Fargues curve and Scholze's 
 theory of diamonds (\cite{ScholzePadic}, \cite{Fargues}). 
 Again, the more classical integral theory in this paper is 
 in a different direction. Nevertheless, it would be interesting to directly compare 
these constructions with ours. 
In another direction, it should also be possible to develop
 a theory of ``relative'' Rapoport-Zink spaces by combining our group theoretic constructions 
 with the theory of relative displays of T. Ahsendorf (see for example \cite{ACZink}).
Then one can compare these with the ``absolute'' Rapoport-Zink spaces of the current paper when the group is given by Weil restriction of scalars (see \cite{RZDrinfeld} for an example of such a comparison).  
 
 We will now briefly describe the contents of the paper: We start with preliminaries
 on Witt vectors, various notions of ``Witt loop schemes'' (these are variations of the Greenberg transform), and a review of the main definitions of Zink's theory of displays.
In \S 3, we define the group theoretic displays ($(G, \mu)$-displays). We discuss several of their
basic properties, study their deformation theory and define a notion of quasi-isogeny.
In \S 4, we give our  group theoretic  definition of the Rapoport-Zink stacks
and state the representability conjecture. 
In \S 5, we show representability for local Shimura data 
of Hodge type over Noetherian rings. 
At the end of the paper, we include three short appendices: The first
reviews certain facts about parabolic subgroups of reductive groups and the second 
discusses torsors for  Witt loop group schemes. Finally, the third appendix  gives some
results on nilradicals of certain rings which are used in the proof of the main representability theorem.

 \smallskip

\noindent{\it Acknowledgment.} We thank M. Hadi Hedayatzadeh, B. Howard, R. Noot, and M. Rapoport for 
useful discussions and suggestions, and the referee for his/her careful reading of the paper.
\bigskip

\section{Preliminaries}\label{sec:pre}

Let $p$ be a prime number. Denote by $k$ an algebraic closure of ${\mathbb F}_p=\Z/p\Z$. Set $W=W(k)$ for the ring of Witt vectors and $K=W[1/p]$ for its fraction field. Denote by $\bar K$ an algebraic closure of $K$. We will use the symbol $k_0$ to denote a finite field of cardinality $q=p^f$ contained in $k$. 

If $\O$ is a $\Z_p$-algebra, we will denote by $\ANilp_\O$
the category of $\O$-algebras in which $p$ is nipotent.
Similarly, we let ${\rm Nilp}_{\O}$ be the category of 
$\Spec(\O)$-schemes $S$ which are such that $p$ is Zariski locally nilpotent on $\O_S$.

In most of the paper, $G$ stands for  a connected 
reductive group scheme  over $ \Z_p $. Its  generic fiber 
is a connected reductive group  over $\q_p$,  and  is unramified, \emph{i.e.}~quasi-split and split over an unramified extension of $\q_p$. 
Conversely, every unramified  connected reductive group   over $\q_p$ is isomorphic to the generic fiber of such a $G$.

\subsection{Witt vectors}

If $R$ is a commutative ring which is a $\Z_p$-algebra, we will denote by $W_n(R)$ the
ring of $p$-typical Witt vectors $(r_0, \ldots , r_{n-1})$ of length $n$ with entries in $R$. We allow in the notation $n=\infty$;
in this case, we simply denote by $W(R)$ the ring of Witt vectors with entries in $R$. For $r\in R$, we set
\[
[r]=(r,0,0,\ldots ).
\]
The ring structure on $W_n(R)$ is functorial in $R$.
Recall the ring homomorphisms (``ghost coordinates'')
\[
w_k: W_n(R)\to R\ ;\ (r_0, r_1,\ldots , r_{n-1})\mapsto r_0^{p^k}+pr_1^{p^{k-1}}+\cdots +p^kr_k.
\]
We denote by $I_n(R)$, or simply $I(R)$ if $n=\infty$, the kernel of $w_0: W_n(R)\to R$.
The Frobenius $F_n$ and Verschiebung $V_n$ are maps $ W_{n+1}(R)\to W_n(R)$, resp.
$W_n(R)\to W_{n+1}(R)$, that satisfy the defining relations
\[
w_k(F_nx)=w_{k+1}(x),\quad w_k(V_nx)=pw_{k-1}(x), \ w_0(V_nx)=0.
\]
The Frobenius $F_n$ is a ring homomorphism.
 The Verschiebung $V_n$ is additive, is given by
 \[
 V_n(r_0,r_1,\ldots , r_{n-1})=(0, r_0, r_1,\ldots ,r_{n-1}),
 \]
 and we  have $V_nW_n(R)=I_{n+1}(R)$. Again, we usually omit the subscript $n$ if $n=\infty$.
 We have the identities
 \[
 F\circ V=p,\qquad V(Fx\cdot y)=x\cdot V(y).
 \]

The following will be used later in the paper. We will denote by $J(R)$ the Jacobson radical of $R$.
\begin{lemma}
\label{jaclemma02}
Let $\mathcal A\subset W_n(R)$ be an ideal, and suppose that one of the following two assertions holds:
\begin{itemize}
\item[(i)]
$p\in J(R)$ and $Rw_0(\mathcal A)=R$.
\item[(ii)]
$n<\infty$  and $Rw_k(\mathcal A)=R$ holds for all $0\leq k<n$.
\end{itemize}
Then  $\mathcal A=W_n(R)$.
\end{lemma}
\begin{proof}
We begin with case (ii), which we handle by induction on $n$: It suffices to check $V^{n-1}([x])=(0,0,\ldots, x)\in\mathcal A$, for any $x\in R$.  (Here, for simplicity, we denote the composition $V_{n-1}\circ V_{n-2}\circ\cdots \circ V_1$ as $V^{n-1}$.) The 
assumption allows us to write the element $x$ as a sum $\sum_{i=1}^ma_iw_{n-1}(f_i)$ where $a_i\in R$ and $f_i\in\mathcal A$. It follows that
\begin{equation*}
{V^{n-1}}([x])=\sum_{i=1}^m{{V^{n-1}}([a_iw_{n-1}(f_i)])}
=\sum_{i=1}^mf_i{{V^{n-1}}([a_i])},
\end{equation*}
 as elements of $W_n(R)$, which gives what we wanted. It remains to consider (i).  
Notice that $w_0$ is surjective, so that $w_0(\mathcal A)=Rw_0(\mathcal A)$. It follows that we can assume $\mathcal A$ is principal, i.e. $\mathcal A=W_n(R)f$, 
and we only need to check that $f$ is a unit in every quotient $W_n(R)/{{(V_{n-1}\circ\cdots\circ V_{n-k})}W_{n-k}(R)}=W_k(R)$ of Witt vectors of finite 
length. (This then also implies the case $n=\infty$.) However, this follows from (ii) together with $p\in J(R)$ and $w_k(f)\equiv w_0(f)^{p^k}{\rm mod}\ p$ for every $k$.
\end{proof}

If $X$ is a scheme over $W(R)$, we will denote by ${}^F X$ the scheme over $W(R)$
 obtained by pulling back via the Frobenius, i.e.
 $$
 {}^F {X}={X}\times_{\Spec(W(R)), F}\Spec(W(R)).
$$

\subsection{Greenberg transforms and Witt loop schemes}

Suppose that $X$ is an affine scheme which is of finite type, resp. of finite presentation, over $W_n(R)$. By \cite[\S 4]{GreenbergI}
(see also \cite[Prop. 29]{Kreidl}), the functor $R'\to X(W_n(R'))$ is represented by an affine scheme $F_nX$ over 
$R$; this is of finite type, resp. of finite presentation, if $n<\infty$. The scheme  $F_nX$ is sometimes called the Greenberg transform of $X$.
(Again, for $n=\infty$, we will simply write $FX$ instead of $F_\infty X$. Also, if $X$ is a scheme over $W(R)$,
we will write $F_nX$ instead of $F_n(X\otimes_{W(R)}W_n(R))$.) 

We can also consider the functor $R'\to X(W(R')[1/p])$. By \cite[Prop. 32]{Kreidl}, we see that this functor is represented by an Ind-scheme $F_{(p)}X$ over $R$ which we might call the Witt loop scheme of $X$. In \cite{Kreidl},   
$F_{(p)}X$ is called the ``localized'' Greenberg transform of $X$. 

We collect a few useful properties of the Greenberg transforms $F_n$.

\begin{proposition}\label{greenbergprop}
a)  If $X$ and $Y$ are two affine finite type schemes over $W_n(R)$, then
there is a natural isomorphism
\begin{equation}\label{Greenbergproduct}
F_n(X\times_{\Spec(W_n(R))}Y)\cong F_nX\times_{\Spec(R)}F_nY.
\end{equation}

b) If $f: X\to Y$ is a formally smooth, resp.  formally \'etale, morphism of  affine schemes over $W_n(R)$, then $F_nf: F_nX\to F_nY$
is formally smooth, resp. formally \'etale, morphism of affine schemes over $R$.

c) If $X$ is affine and smooth  over $W_n(R)$, and $n<\infty$, then $F_nX$ is  smooth over $R$.

d) If $X$ is  affine and smooth   over $W(R)$ ($n=\infty$), then $FX$ is flat and formally smooth over $R$.

\end{proposition}

\begin{proof}
Part (a) follows quickly from the defining property of the Greenberg transforms $F_nX(R')=X(W_n(R'))$,
$F_nY(R')=X(W_n(R'))$. Let us show part (b). Consider an $R$-algebra $B$ with a nilpotent ideal $I\subset B$ and set $\bar B=B/I$. Then there is a natural map of sets
\[
a: F_nX(B)\to F_nY(B)\times_{F_nY(\bar B)}F_nX(\bar B).
\]
By definition, $F_nf$ is formally smooth, resp. formally \'etale, if $a$ is surjective, resp. bijective,
for all such pairs $I\subset B$. By definition, the map $a$ is the natural map
\[
a: X(W_n(B))\to Y(W_n(B))\times_{Y(W_n(\bar B))}X(W_n(\bar B)).
\]
However, we can easily see that the kernel $W_n(I)$ of $W_n(B)\to W_n(\bar B)$ 
is still nilpotent (also for $n=\infty$). In fact, if $I^r=0$, then $W_n(I)^r=0$. Hence, the surjectivity,
resp. bijectivity, of $a$ follows since $f$ is assumed to be formally smooth, resp. formally \'etale.
Part (c) follows from (b) and the above since then $F_nX\to \Spec(R)$ is of finite presentation, and so smooth amounts to formally 
smooth. Finally, to show (d) observe that, under our assumptions,  the natural morphism $F_{n+1}X\to F_nX$ 
is formally smooth, for all $n$. (This is obtained using the fact that if $\bar B=B/I$ with $I$ nilpotent, then the natural homomorphism $W_{n+1}(B)\to W_n(B)\times_{W_n(\bar B)}W_{n+1}(\bar B)$ is surjective with nilpotent kernel.) If we write $A_n$ for the $R$-algebra
with $F_n(X\otimes_{W(R)}W_n(R))=\Spec(A_n)$, we have $FX=\Spec(\varinjlim_{n} A_n)$.
Formal smoothness of $FX$ over $R$ follows as above.
Flatness also follows from (c) since smooth implies flat, using also that 
a direct limit of flat $R$-algebras is $R$-flat.
\end{proof}
 
\subsubsection{}
If $X$ is an affine finite type scheme over $W(k_0)$, we will  write $L^+_nX$ and $LX$ for the Greenberg and localized Greenberg transforms of the base changes of $X$ to $W_n(W(k_0))$ by the natural Cartier
ring homomorphism $\Delta_n: W(k_0)\to W_n(W(k_0))$ characterized by $w_k(\Delta_n(x))=F^k(x)$. 
These Greenberg transforms are schemes, resp. an Ind-scheme, over $W(k_0)$.

Since $w_0\cdot \Delta$ is the identity, there is a natural morphism
\[
L^+X\to X
\]
induced by $w_0: W(R)\to R$. We will denote  by $s_0\in X(R)$ the image of the point $s\in L^+X(R)=X(W(R))$
under this map.

The homomorphism $\Delta_n$  commutes with the Frobenius
on the source and target (\cite[Lemma 52]{Zinkdisplay}) and we can see that
we have natural isomorphisms
\[
{}^F (L^+X) \simeq L^+ ({}^F X), \quad {}^F(LX)\simeq L ({}^F X),
\]
while $F$ induces natural morphisms
\[
F: L^+X\to {}^F (L^+X), \quad F: LX\to {}^F(LX)
\]
which cover the Frobenius isomorphism of $\Spec(W(k_0))$.

If $X$ is in addition a group scheme over $W(k_0)$, then $L^+_nX$, resp. $LX$, are group schemes, resp. is a Ind-group scheme, over $W(k_0)$. In this case, the isomorphisms and morphisms above are group scheme, resp. Ind-group scheme, homomorphisms.

\begin{remark}
Suppose that $R$ is a $k_0$-algebra via $\phi: k_0\to R$ and $X=\Spec(A)$ an affine finite type $W(k_0)$-scheme.
We view $R$ as a $W(k_0)$-algebra via the composition $W(k_0)\to k_0\to R$ where the first
map is $w_0$. By our definition of the $W(k_0)$-scheme $L^+X$ above, which uses the Cartier homomorphism, 
the points $L^+X(R)=X(W(R))$ are given by ring homomorphisms $A\to W(R)$ such that the composition
$W(k_0)\to A\to W(R)$
is equal to
\[
W(k_0)\xrightarrow{\Delta} W(W(k_0))\xrightarrow{\epsilon}W(k_0)\xrightarrow{W(\phi)}W(R),
\]
where the second map $\epsilon$ is the result of applying the functor $W(-)$ to $w_0: W(k_0)\to k_0$.
However, $\epsilon\circ\Delta$ is the identity by \cite[eq. 92]{Zinkdisplay} and so this composition 
is equal to $W(\phi): W(k_0)\to W(R)$. 

This gives a simpler description of the special fiber $L^+X\otimes_{W(k_0)}k_0$; its $R$-valued points for $\phi: k_0\to R$ are the $W(R)$-valued points of
$X$ where $W(R)$ is regarded as a $W(k_0)$-algebra via $W(\phi)$. In particular, we see that 
\[
FX=L^+X\otimes_{W(k_0)}k_0.
\]
As mentioned above, some authors call
the special fiber $FX=L^+X\otimes_{W(k_0)}k_0$ the Greenberg transform of $X$.
\end{remark}

\subsection{Displays}\label{ZinkDisplays}

Let us now quickly review the definition of displays  as in
\cite{Zinkdisplay}. Suppose that $R$ is a (commutative) ring 
which is $p$-adically complete and separated.
A display over $R$ is a quadruple
\[
\D =(M,N,F_0,F_1)
\]
where $M$ is a finitely generated projective $W(R)$-module, $N$ a submodule such that $I(R)M\subset N$ and
 $M/N$ a projective $R$-module, $F_0 :M\to M$ and $F_1 :N\to M$ are $F$-linear maps such that 
the image $F_1(Q)$ generates $M$ as a $W(R)$-module, and we have $F_1(Vw\cdot x) = wF_0(x)$ for $w\in W(R)$ and 
$x\in M$. See \cite{Zinkdisplay} and other places for the definition of nilpotence and of a nilpotent display.  
To avoid confusion, let us note that in \cite{Zinkdisplay} displays are called $3n$-displays ($3n=$ not necessarily nilpotent), while nilpotent displays are called displays. The notation for these objects there is $(P, Q, F, V^{-1})$ in which 
$V^{-1}$ is just a symbol. 

Over a perfect field, a display is the same as a Dieudonn\'e module
$(M, F, V)$; then $N = V(M)$ and $F_1$ is the inverse of $V$. In that case,
the nilpotence condition means that $V$ is $p$-adically topologically nilpotent.

Displays over $R$ form a category and Zink constructs a functor $BT$ from nilpotent displays 
over $R$ to formal $p$-divisible groups over $R$ which he shows to be an equivalence of categories
in many cases, for example when $R/pR$ is a finitely generated algebra over 
${\mathbb F}_p$. Later, Lau \cite{LauInvent} showed that $BT$ gives an equivalence
  for all $R$ which are $p$-adically complete and separated.

Let $\D=(M,N,F_0,F_1)$ be a display over $R$ and suppose there exists a ``normal decomposition'' of $\D$, i.e.
a decomposition $M=T\oplus L$ with $T$, $L$ two $W(R)$-modules such that also $N=I(R)T\oplus L$. 
Suppose that both $T$ and $L$ are free $W(R)$-modules of rank $d$ and $h-d$
(we can always find such a decomposition Zariski locally on $\Spec(R)$, see \cite{Zinkdisplay}). Assume $e_1,\ldots , e_d$
is a basis of $T$ and $e_{d+1},\ldots , e_h$ a basis of $L$. Then there is an invertible matrix $U=(u_{ij})$
in ${\rm GL}_h(W(R))=L^+\GL_h(R)$ such that
\[
F_0e_j=\sum_{i=1}^h u_{ij}e_i, \ \ \ 1\leq j\leq  d,
\]
\[
\ \ F_1e_j=\sum_{i=1}^h u_{ij} e_{i}, \ \ \ d+1\leq j\leq  h.
\]
We can write this as a block matrix
\[
U=\begin{pmatrix}
A & B\\ C& D
\end{pmatrix}
\]
with $A$ of size $d\times d$ and $D$ of size $(h-d)\times (h-d)$. Suppose that $\D'$ is another display over $R$
which is also given by a block matrix $U'$ with blocks of the same sizes. Then a morphism of displays $\D\to \D'$ is given by a
block matrix of the form
\[
H=\begin{pmatrix}
X & V(Y)\\ Z& T
\end{pmatrix}
\]
with $X$, $Y$, $Z$, $T$ blocks with coefficients in $W(R)$ which satisfies
\begin{equation}\label{PhiId1}
\begin{pmatrix}
A' & B'\\ C'& D'
\end{pmatrix} \begin{pmatrix}
F(X) & Y\\ pF(Z)& F(T)
\end{pmatrix}=\begin{pmatrix}
X & V(Y)\\ Z& T
\end{pmatrix}\begin{pmatrix}
A & B\\ C& D
\end{pmatrix}.
\end{equation}
 Set
 \begin{equation}\label{Phi}
 \Phi(H)=\Phi\left(\begin{pmatrix}
X & V(Y)\\ Z& T
\end{pmatrix}\right)=\begin{pmatrix}
F(X) & Y\\ pF(Z)& F(T)
\end{pmatrix}.
 \end{equation}
 The morphism $\D\to \D'$ is an isomorphism if and only if $H$ is invertible. Then the identity above can be written
 \begin{equation}\label{PhiId}
 H^{-1}U'\Phi(H)=U.
 \end{equation}
 
 By \cite[Theorem 37]{Zinkdisplay} displays form a fpqc stack over ${\rm Nilp}_{\Z_p}$. 
 The above discussion then implies that displays of rank $h$ and ${\rm rank}_R(M/N)=d$ are given by the fpqc quotient stack
\[
 [L^+\GL_h /_{\Phi}\,   H^{(d,h-d)}]
 \]
 over ${\rm Nilp}_{\Z_p}$. Here, $H^{(d,h-d)}(R)$ is the subgroup of $L^+\GL_h(R)=\GL_h(W(R))$ of matrices of the
 form 
 \[
H=\begin{pmatrix}
X & V(Y)\\ Z& T
\end{pmatrix}
\]
as above and the quotient is for the right action by ``$\Phi$-conjugation''  as in (\ref{PhiId}).
Of course, here $\Phi$ is the $F$-linear map given by (\ref{Phi}).

\section{$(G,\mu)$-displays}\label{sec:Gdisplays}

In this section, we define $(G,\mu)$-displays and show several basic properties.

\subsection{The divided Frobenius $\Phi_{G, \mu}$}

We start by defining the ``divided Frobenius''. This generalizes the map $\Phi$
of the previous section and plays a central role in everything that follows.

\subsubsection{}
Suppose $G$ is a reductive group scheme over $\Z_p$
and that 
\[
\mu: \Gm_{W(k_0)}\to G_{W(k_0)}
\]
 is a minuscule cocharacter.
As in Appendix \ref{simplerversion} we will denote by  $P_\mu\subset G_{W(k_0)}$ the parabolic subgroup scheme defined by $\mu$. This is  the parabolic subgroup of $G$ 
such that $P_\mu\times_{W(k_0)}W$ contains exactly the root groups $U_a$ of the split group $G_W$, 
for all roots $a$ with $\langle \mu, a\rangle \geq 0$.   We denote by $U_\mu$ the corresponding unipotent 
group which is the unipotent radical of $P_\mu$.

We will denote by $\frakg$,   $\frakp$, $\frakp^{-}$, $\fraku$, $\fraku^{-}$, the Lie algebras 
of $G$,   $P_\mu$, $P_{\mu^{-1}}$, $U_\mu$, $U_{\mu^{-1}}$; these are finite free $\z_p$-, resp.  $W(k_0)$-modules,
and we have the weight decompositions
\begin{equation}\label{Liealgebras}
W(k_0)\otimes_{\z_p} \frakg=\frakp\oplus\fraku^-.
\end{equation}

  We will denote by $H^\mu$  the group scheme  over $\Spec(W(k_0))$
  with
  $$
    H^\mu(R)=\{ g\in L^+G(R)\ |\ g_0\in P_\mu(R)\}.
  $$
   We can see that  $H^\mu$  is a closed subgroup scheme of $L^+G$. 
 \begin{proposition}\label{divFrobProp}
There is a group scheme homomorphism 
$$
\Phi_{G,\mu}: H^\mu\to L^+   G_{W(k_0)}
$$
characterized by the following property: We have
\begin{equation}\label{identDivFr}
\Phi_{G,\mu}(h)= F\cdot (\mu(p)\cdot h\cdot \mu(p)^{-1})\in {}^FG(W(R)[1/p])=G(W(R)[1/p]).
\end{equation}
\end{proposition}

 \begin{proof}
 Consider  the group scheme $L^{>0}U_{\mu^{-1}}\subset L^+U_{\mu^{-1}}$ with $R$-valued points
  $u\in U_{\mu^{-1}}(W (R))$ such that $u_0=1$. 
  
  We will use the following:
  
  \begin{proposition}\label{productI}
  Assume that $p$ is in the Jacobson radical $J(R)$ of $R$. Then multiplication in $L^+G(R) $ gives a bijection
$$
 L^+P_\mu(R)\times L^{>0}U_{\mu^{-1}}(R)\cong H^\mu(R).
 $$
   \end{proposition}
   \begin{proof}  
   Let $h\in H^\mu(R)\subset L^+G(R)=G(W(R))$. 
 We would like to show that the corresponding morphism $h: \Spec(W (R))\to G$ factors through $G^*=P_\mu\times_{W(k_0)}U_{\mu^{-1}}$. 
 Recall that by \ref{parabolicfacts}, $G^*$ is an open subscheme of the affine scheme $G=\Spec(A_0)$ with the open immersion $G^*\to G$ given by 
 multiplication. Suppose $\sqrt\mathcal I=\mathcal I\subset A_0$ is the ideal corresponding to the reduced induced subscheme structure on the complement 
  $G-G^*$. The element $h$ is given by $h^*:  A_0\to W(R)$. Consider  the composition
 $$
 h^*_0=w_0\circ h^*: A_0\to W(R)\xrightarrow{ } R.
 $$
 We know that $h^*_0(\mathcal I)R=R$. By part (i) of  
Lemma \ref{jaclemma02} it follows that $h^*(\mathcal I)W(R)=W(R)$, and so 
$h$ factors through $G^*$. The result now
follows from the definition of $H^\mu$. 
\end{proof}
 
We now continue with the construction of $\Phi_{G,\mu}$. By Lemma \ref{minusculeunipotent},
we have $U_{\mu^{-1}}\simeq 
{\mathbb G}_a^r\times_{\Z_p}W(k_0)$. We have $L^+{\mathbb G}^r_a(R)=W(R)^r$,
$L^{>0}{\mathbb G}^r_a(R)=I(R)^r$ and we can define
\[
V^{-1}: L^{>0}{\mathbb G}^r_a\to {}^FL^+{\mathbb G}^r_a
\]
as given by the direct sum of $r$ copies of $V^{-1}: I(R)\to W(R)$. This
 gives
  $$
  V^{-1}: L^{>0}U_{\mu^{-1}}\to {}^FL^+U_{\mu^{-1}}
  $$
  which is independent of choices.
  Consider also the composition
  $$
 F\cdot \Int_\mu(p): L^+P_\mu(R)\to L^+P_\mu(R) \to {}^FL^+{} G(R)=L^+{}G(R).
  $$
 Here  we are using the extension of conjugation
  $\Int_\mu(p): {\mathbb A}^1\to {\rm End}(P_\mu)$ (see \ref{dynamicparabolic}) applied to $W (R)$
  so that $\Int_\mu(p)\in {\rm End}(P_\mu)(W (R))$.
  
First suppose that $p\in J(R)$ and let $h\in H^\mu(R)$;
by  Proposition \ref{productI}, we can write (uniquely) 
$$
h=h'\cdot h'', \quad h'\in L^+P_\mu(R), \ \ h''\in L^{>0}U_{\mu^{-1}}(R),
$$
and we set
$$
\Phi_1(R)(h):=(F\cdot  \Int_\mu(p))(h')\cdot V^{-1}(h'') \in {}^FL^+G(R)=L^+G(R).
$$
This gives 
$$
\Phi_1(R): H^\mu(R)\to L^+G(R)
$$
when $p\in J(R)$.

Now suppose $R$ is arbitrary; consider $R_1=(1+pR)^{-1}R$ and $R_2=R[1/p]$.  
This gives a faithfully flat cover of $\Spec(R)$  
  $$
  \Spec(R)=\Spec(R_1)\cup \Spec(R_2).
  $$
Notice that $p$ is in every maximal ideal of $R_1$; indeed if  
 $\frakm$ is such an ideal and $p$ is not in $\frakm$, we can find $r=b/(1+pb')\in R_1$,
 such that $pr-1\in \frakm$. This gives $1+p(b'-b)\in \frakm$, a contradiction,
 since this is a unit in $R_1$. Therefore, $p$ is in the Jacobson radical of $R_1$.

\begin{itemize}

\item
Let   $h_2\in H^\mu(R_2)=H^\mu(R[1/p])$. Since 
\[
p\in W(R[1/p]))^\times=\g_m(W(R[1/p])),
\]
we can consider $\mu(p)$, $\mu(p)^{-1} \in L^+G(R[1/p])$. We  define 
 $$
 \Phi_2(R_2)(h_2):=F(\mu(p)\cdot h_2\cdot \mu(p)^{-1})\in L^+G(R[1/p]).
 $$
 \item Let $h_1\in H^\mu(R_1)=H^\mu((1+pR)^{-1}R)$. 
 Since $p\in J(R_1)$, we consider $\Phi_1(R_1)(h_1)\in L^+G(R_1)$ as above.
 \end{itemize}
 
 Now let us apply this to $R=\O_{H^\mu}$ the affine algebra of $H^\mu$ and the universal points $h_i\in H^\mu(R_i)$.
We obtain elements $\Phi_i(R_i)(h_i)\in L^+ G(R_i)$.  
 These points  agree over $\Spec(R_{12})$ with $R_{12}=(1+pR)^{-1}R\otimes_R R[1/p]=  (1+pR)^{-1}R[1/p]$.
Indeed, when $p$ is invertible, the map $V^{-1}: L^{>0}U_{\mu^{-1}}\to {}^FL^+ U_{\mu^{-1}}$ 
  agrees with division by $p$ followed by $F$. On the other hand, the  adjoint action of $\mu(p)^{-1}$ on $\fraku^-=\Lie(U_{\mu^{-1}})$
  agrees with multiplication by $p$. Therefore, by descent, we obtain a well-defined morphism
  \begin{equation} 
 \Phi_{G,\mu} :  H^\mu\to L^+ G_{W(k_0)}.
  \end{equation}
Notice that $R=\O_{H^\mu}$ is $p$-torsion free and so $R\subset R_2=R[1/p]$
and $W(R)\subset W(R[1/p])$. Since $\Phi_2(R_2)$ is obviously a group homomorphism
we conclude that $\Phi_{G,\mu}$ is a group scheme homomorphism which satisfies the 
identity (\ref{identDivFr}). In fact, by considering $R=\O_{H^\mu}$ 
we see that $\Phi_{G,\mu}$ is the unique morphism that satisfies that identity.
\end{proof}

\subsection{Definitions and basic properties}

Suppose that $(G, \mu)$ is a pair of a reductive group scheme over 
$\z_p$ and a minuscule cocharacter $\mu: \g_{W(k_0)}\to G_{W(k_0)}$, 
where $k_0$ is a finite field of characteristic $p$. In the previous 
paragraph, we have constructed $H^\mu$ and 
\[
\Phi_{G,\mu}: H^\mu\to {}^FL^+G_{W(k_0)}=L^+G_{W(k_0)}
\]
over $\Spec(W(k_0))$.  Suppose $S$ is a $W(k_0)$-scheme.

\begin{definition}
A $(G,\mu)$-display over $S$ is a triple $\D:=(P, Q, u)$ where
\begin{itemize}
\item $Q$ is an (fpqc locally trivial) $H^\mu$-torsor   over $S$,
\item $P:=Q\times_{H^\mu}L^+G_{W(k_0)}$ is the induced $L^+G_{W(k_0)}$-torsor,
\item $u:  Q\to P$ is a morphism which is compatible with 
  $\Phi_{G,\mu} $ in the sense that $u(q\cdot h)=u(q)\cdot \Phi_{G,\mu}(h)$.
\end{itemize}
\end{definition}
Notice that $P$ is determined by $Q$ by $P:=Q\times_{H^\mu}L^+G_{W(k_0)}$ 
and so we will sometimes omit it from the notation. Our convention is that groups act on
the right.

A morphism $(P_1, Q_1, u_1)\to (P_2, Q_2, u_2)$ between two $(G,\mu)$-displays
is a $H^\mu$-torsor isomorphism $Q_1\xrightarrow{\sim} Q_2$ which is compatible with
$u_1$ and $u_2$ in the obvious manner. 

In most of the paper, we consider $(G,\mu)$-displays over schemes in ${\rm Nilp}_{W(k_0)}$.
We can see that $(G, \mu)$-displays form a fpqc stack in groupoids over ${\rm Nilp}_{W(k_0)}$ 
which we will denote by $\B(G,\mu)$.

\begin{example}\label{glnexample}
{\rm For non-negative integers $d\leq h$, we let $\mu_{d,h}: \g_m\to \GL_h$ be the minuscule cocharacter 
of $\GL_h$ over $\z_p$ given by\footnote{The notation $z^{(r)}$ means that there are $r$ copies of $z$.}
\begin{equation*}
\mu_{d, h}(z)=\diag( {1}^{(d)},{z }^{(h-d)}).
\end{equation*}
The discussion in \ref{ZinkDisplays} implies that there is an equivalence between 
 the stack of $(\GL_h, \mu_{d, h})$-displays, and the stack of (Zink, not-necessarily-nilpotent) displays of rank $h$ and 
dimension $d$. }
\end{example}

\subsubsection{} The fpqc quotient $L^+G_{W(k_0)}/H^\mu$ can be identified with the homogeneous space $X_\mu=G_{W(k_0)}/P_\mu$. If $\D:=(P, Q, u)$ is a 
$(G, \mu)$-display over $S$ then the quotient $P/H^\mu$ is   a $X_\mu$-bundle. Since the $L^+G_{W(k_0)}$-torsor
$P$ admits a reduction (given by $Q$) to the subgroup $H^\mu$  we also obtain a section
\[
\alpha: S=Q/H^\mu\to P/H^\mu
\]
of this $X_\mu$-bundle over $S$. We call this section the Hodge filtration of the $(G, \mu)$-display $\D$.
We will also consider the vector bundle 
\[
T_\D:=\alpha^*({\mathcal N}_{\alpha(S)|P/H^\mu})
\]
 over $S$ obtained by pulling
back via $\alpha$ the normal bundle of the (regular) closed immersion
$\alpha(S)\hookrightarrow P/H^\mu$. Under a   condition on $\D$, we will
see that the bundle $T_\D$ controls the  deformations of $\D$
 (see Theorem \ref{globaldef}).
 
 \subsubsection{}\label{torsordatum} Suppose $S=\Spec(B)$ with 
 $B$ a $p$-adically complete and separated $W(k_0)$-algebra.
Then using Proposition \ref{torsor02} (see also Remark \ref{torsor04} (ii)) we can reinterpet the datum of the $H^\mu$-torsor $Q$ as a pair $(\P, \alpha)$,
where $\P$ is a $G$-torsor over $\Spec(W(B))$
and $\alpha$ a section over $\Spec(B)$ of the $X_\mu$-bundle
obtained by first restricting $\P$ along the closed immersion $\Spec(B)\to \Spec(W(B))$ and then taking quotient
by the action of $P_\mu$. The $G$-torsor $\P$ corresponds, via 
Appendix \ref{torsor02}, to $P:=Q\times_{H^\mu}L^+G_{W(k_0)}$ so that $P=F\P$, the $X_\mu$-bundle is $P/H^\mu$ and the section is $Q/H^\mu\to P/H^\mu$
as above.

\begin{definition}
  We say that a $(G,\mu)$-display $\D=(P, Q, u)$  
is banal, if the torsor $Q$ is trivial. 
Banal $(G,\mu)$-displays over $S$ in ${\rm Nilp}_{W(k_0)}$ give a full 
subgroupoid $B(G, \mu)(S)$ of $\B(G, \mu)(S)$.
\end{definition}

\begin{remark} 
By Corollary \ref{torsCor}, any (fpqc locally trivial) $H^\mu$-torsor or $L^+G$-torsor over $S$ 
in ${\rm Nilp}_{W(k_0)}$ is 
locally trivial for
the \'etale topology on $S$. Therefore, every $(G,\mu)$-display  
is banal locally for the \'etale topology on $S$.
\end{remark}

\subsubsection{} Suppose that $(P, Q, u)$ is banal. Then,
after choosing a trivialization
\[
\alpha: H^\mu\xrightarrow{\sim} Q
\] which also induces 
$\alpha:  L^+G_{W(k_0)}\xrightarrow{\sim} P:=Q\times_{H^\mu}L^+G_{W(k_0)}$, the triple 
 $(P, Q, u)$ is determined by 
\[
U:=\alpha^{-1}u(\alpha(1))\in {}^FL^+G(\O_S(S))=G(W(\O_S(S))).
\]
 A different trivialization $\alpha'=\alpha\cdot h$ gives $U'=\alpha'^{-1}u(\alpha'(1))$.
 We can see that  
 $$
 U'=h^{-1}\cdot U\cdot \Phi_{G,\mu}(h).
 $$
Therefore, the groupoid of banal $(G,\mu)$-displays $B(G, \mu)(S)$ can be identified with the quotient
groupoid
$$
[L^+G(S)/_{\Phi_{G,\mu}} \ H^\mu(S)]
$$
where the action is by $\Phi_{G,\mu}$-conjugation as above. This implies that
the fpqc stack of $(G, \mu)$-displays  can   be identified with a fpqc quotient
stack 
\[
\B(G,\mu)=[L^+G_{W(k_0)}/_{\Phi_{G,\mu}} \ H^\mu].
\]

\subsubsection{}
By the definition, there is a natural forgetful morphism 
\[
q: \B(G,\mu)\to BH^\mu, \quad (P, Q, u)\mapsto Q,
\]
where $BH^\mu$ is the fpqc stack of $H^\mu$-torsors.

 \begin{lemma}\label{diagonal}
 a) The morphism $q: \B(G,\mu)\to BH^\mu$ is representable
and affine. 

b) The diagonal morphism $\Delta: \B(G,\mu)\to \B(G, \mu)\times_{W(k_0)}\B(G, \mu)$ is representable and affine.
\end{lemma}

\begin{proof}
If $Q_i\to S$, $i=1,2$, are $H^\mu$-torsors, 
we can see by descent, that the
functor on $S$-schemes 
\[
(T\to S)\mapsto {\rm Isom}_{H^\mu}(Q_1\times_S T, Q_2\times_S T)
\]
is represented by an affine $S$-scheme $\underline{\rm Isom}_{H^\mu}(Q_1, Q_2)$. This implies that the diagonal
of the stack $BH^\mu$ is representable and affine. There is a similar statement for isomorphisms between
$L^+G$-torsors.

Now suppose that $S\to BH^\mu$ is the morphism corresponding to an $H^\mu$-torsor $Q$. Recall that we have
$P=Q\times_{H^\mu} L^+G$.
Then the fibered 
product $S\times_{BH^\mu}\B(G,\mu)$ is represented by the affine $S$-scheme
$\underline{\rm Isom}_{L^+G}(P', P)$, where $P'=Q\times_{H^\mu, \Phi_{G,\mu}} L^+G$.
Part (a) now follows easily from the definition of $\B(G, \mu)$; then part (b) follows quickly from part (a) and the above.
\end{proof}

\subsubsection{} Suppose that $A$ is a Noetherian $W(k_0)$-algebra complete and separated 
for the $I$-adic topology for $I\subset A$ an ideal that contains a power of $p$. Then $A$ is also complete and separated
for the $p$-adic topology. 

\begin{proposition}\label{limitdisplay}
There is a natural equivalence between the category of $(G,\mu)$-displays $\D$ over $A$ and the category of compatible 
systems of $(G, \mu)$-displays $\D_n$ over $A/I^n$, $n\geq 1$, given by $\D\mapsto \{\D\times_A A/I^n\}_n$.
\end{proposition}

\begin{proof}
The full faithfulness of the functor $\D\mapsto \{\D\times_A A/I^n\}_n$ follows easily 
by using Lemma \ref{diagonal} (b). Let us show essential surjectivity: Consider a compatible 
sequence of $(G, \mu)$-displays $\D_n=(P_n, Q_n, u_n)$ over $A/I^n$; we would like to construct $\D=(P, Q,  u)$ over $A$. Using Lemma \ref{torsor03} (b) 
and Remark \ref{torsor04} (i) we can construct a $L^+G$-torsor $P$ over $A$ with compatible 
isomorphisms $P\times_A A_n\simeq P_n$. To give the $H^\mu$-torsor $Q$ over $A$
we use Remark \ref{torsordatum} and apply Grothendieck's algebraization theorem 
to the proper morphism $P/H^\mu\to \Spec(A)$. Finally, the homomorphism $u: Q\to P$ is given 
from $\{u_n\}_n$ using that $\underline{\rm Isom}_{L^+G}(Q\times_{H^\mu, \Phi_{G,\mu}} L^+G, P)$
is affine (see also the proof of \ref{diagonal} (b) above). 
\end{proof}

\subsubsection{}
Suppose that $\mu': \g_{m\,W(k_0)}\to G_{W(k_0)}$ is conjugate to $\mu$, i.e. 
$\mu'={\rm Int}(g)\circ\mu$, for $g\in G(W(k_0))$. Then ${\rm Int}(g): H^\mu\xrightarrow{\sim} H^{\mu'}$
and 
$$
\Phi_{G,\mu'}={\rm Int}({^Fg})\cdot \Phi_{G,\mu}\cdot{\rm Int}(g)^{-1}.
$$
Using this, we see that conjugation by $g$ gives an isomorphism $\B(G,\mu)\xrightarrow{\sim}\B(G,\mu')$, cf. \cite[\S 3.3.2]{E7}.

\subsubsection{}\label{twopairs} Suppose that $(G_i,\mu_i)$, $i=1$, $2$, are two pairs as 
above and that we are given a group scheme homomorphism $\rho: G_1\to G_2$   such that 
$\mu_2=\rho(\mu_1)$. We then write $\rho: (G_1,\mu_1)\to (G_2,\mu_2)$. Then 
$\rho$ induces group scheme homomorphisms $\rho: H^{\mu_1}_1\to H^{\mu_2}_2$ 
and $\rho:L^+G_1\to L^+G_2$
and we have
$
\rho\cdot \Phi_{G_1,\mu_1}=\Phi_{G_2,\mu_2}\cdot \rho$. Using this we obtain a morphism 
$$
\rho_*: \B(G_1,\mu_1)\to \B(G_2,\mu_2).
$$ (cf. \cite[\S 3.3.1]{E7}).

\subsubsection{}\label{displaysAlgCl} Let us discuss $(G,\mu)$-displays over the algebraically closed field $k$. 
Since $I(k)=VW(k)=pW(k)$, we have
\[
H^\mu(k)=\mu^{-1}(p)G(W)\mu(p)\cap G(W)
\]
and 
\[
\Phi_{G, \mu}(h)=F(\mu(p)h\mu^{-1}(p))=\mu^\sigma(p)\sigma(h)\mu^\sigma(p)^{-1}.
\]
\begin{proposition}\label{displaysoverk}
The set of isomorphism classes of $(G,\mu)$-displays over  $k$ is in 
$1$-to-$1$ correspondence with the quotient
\[
G(W)\mu^\sigma(p)G(W)/_{_\sigma}G(W)
\]
by $\sigma$-conjugation.
\end{proposition}

\begin{proof}
Since both torsors $Q$ and $P$ are  trivial, a $(G, \mu)$-display $\D=(P, Q, u)$ over $k$ is given by $u\in G(W)$. The action by $\Phi_{G,\mu}$-conjugation is given by
\[
u\mapsto h^{-1}\cdot u\cdot \Phi_{G,\mu}(h)= h^{-1}u\mu^\sigma(p)\sigma(h)\mu^\sigma(p)^{-1}.
\]
If we set $g=u\mu^\sigma(p)$, we see that this action replaces $g$ by $h^{-1}g\sigma(h)$. Therefore,
the isomorphism classes of $(G,\mu)$-displays over $k$ are in 
$1$-to-$1$ correspondence with the quotient
\[
G(W)\mu^\sigma(p)/_{_\sigma}  (\mu(p)^{-1}G(W)\mu(p)\cap G(W))
\]
by $\sigma$-conjugation. Consider the natural map
\[
G(W)\mu^\sigma(p)/_{_\sigma}(\mu(p)^{-1}G(W)\mu(p)\cap G(W))\to G(W)\mu^\sigma(p)G(W)/_{_\sigma}G(W).
\]
If $g\mu^\sigma(p)=h^{-1}g'\mu^\sigma(p)\sigma(h)$ for $h\in G(W)$, then 
$g'^{-1}hg=\mu^\sigma(p)\sigma(h) \mu^\sigma(p)^{-1}$. This implies 
$\sigma(h)\in \mu^\sigma(p)^{-1}G(W)\mu^\sigma(p)$, so $h\in  \mu(p)^{-1}G(W)\mu(p)\cap G(W)$.
This shows that the above map is injective. The map is also surjective: consider $u_1\mu^\sigma(p)u_2$
and write $u_2=\sigma(h)^{-1}$, for $h\in G(W)$. Then we have
\[
h^{-1}u_1\mu^\sigma(p)u_2\sigma(h)\in G(W)\mu^\sigma(p)
\]
which shows $u_1\mu^\sigma(p)u_2$ is in the image.
\end{proof}

\subsubsection{} Notice that the above proof also shows that, if $\D$ is a $(G,\mu)$-display over $k$ given by $u\in G(W)$,
then the $\sigma$-conjugacy class of $b:=u\mu^\sigma(p)\in G(W[1/p])$ only depends 
on the isomorphism class of $\D$. The pair $(\mu, b)$ gives a filtered   
$G$-isocrystal over $W[1/p]$ as considered in \cite[Chapter 1]{RapZinkBook}.

We denote by $\mathbf D$ the algebraic torus over $\q_p$ whose character group
is $\q$. Kottwitz \cite{KottIso} associates to the Frobenius conjugacy class $b$ a morphism of algebraic groups defined over $K:=W[1/p]$
\[
\nu_b : {\mathbf D}_{K}\to G_{K}
\]
called the slope morphism. 

If $\rho: G\to \GL(V)$ is a $\q_p$-rational representation of $G$, the morphism $\rho\cdot \nu_b$
defines a $\q$-grading on the vector space $V\otimes_{\q_p} K$. The morphism $\nu_b$ is characterized by the property that this grading is the slope 
decomposition of the isocrystal over $K$ associated to $b$ and $V$. The rational number $\lambda$ is called a slope of $\rho\cdot\nu_b$ if the corresponding 
isotypic component of $V\otimes_{\q_p} K$ is not equal to zero. In particular, we can consider the adjoint representation $\rho={\rm Ad}^G: G\to \GL({\rm Lie}(G))$ 
and the slopes of ${\rm Ad}^G(b)$.

\subsection{Quasi-isogenies}

\subsubsection{}  Suppose that $S$ is a  scheme over $\Spec({\mathbb F}_p)$ and denote by $F: S\to S$ the absolute
Frobenius. Recall that if $Y\to S$ is an $S$-scheme (or Ind-scheme)
we set
\[
^F Y=Y\times_{S, F}S.
\]
Since $G$ is defined over $\Z_p$ we have a natural isomorphism $^FLG\xrightarrow{\simeq} LG$.
\begin{definition}
Suppose that $S$ is a  scheme over $\Spec({\mathbb F}_p)$. 

A $G$-isodisplay over $S$ is a pair $(T, \phi)$ of a $LG$-torsor $T$ over $S$
and an isomorphism $\phi: {}^F T\xrightarrow{\sim} T$. 

An (iso)morphism
$\alpha: (T_1, \phi_1)\xrightarrow {\sim}  (T_2, \phi_2)$ of two $G$-isodisplays over $S$
is an isomorphism of $LG$-torsors
$\alpha: T_1\xrightarrow {\sim} T_2$
which is compatible with $\phi_1$, $\phi_2$, in the sense that:
$ 
\alpha\cdot \phi_1=\phi_2\cdot {}^F\alpha.
$ 
\end{definition}

\subsubsection{}\label{Gisodisplay} 

Let us now show how to associate a $G$-isodisplay 
\[
\D[1/p]=(P[1/p], \phi)
\]
 to a $(G, \mu)$-display $\D=(P, Q, u)$ over $S\in {\rm Nilp}_{W(k_0)}$.

We first assume $p\cdot \O_S=0$. If $P$ is a $L^+G$-torsor over $S$ in ${\rm Nilp}_{W(k_0)}$, 
then we set 
\[
P[1/p]:=P\times_{L^+G}LG 
\]
for the $LG$-torsor over $S$ obtained 
from $P$ using the natural $L^+G\to LG$. Then if $(P, Q, u)$ is a $(G, \mu)$-display over $S$,
we construct an isomorphism
\[
\phi:\ ^FP[1/p] \xrightarrow{\sim} P[1/p]
\]
of $LG$-torsors over $S$. We can assume that $S$ is affine $S=\Spec(R)$;
the general case can then be handled by descent as below.

Suppose first that $Q$ is the trivial $H^\mu$-torsor and choose a trivialization 
of $Q$ so that $u$ is given by $U\in G(W(R))$.
Both $P[1/p]$ and ${}^FP[1/p]$ are then identified with the trivial $LG$-torsor.
Define
$
\phi:  {}^FP[1/p]\xrightarrow{\sim}  P[1/p]
$
to be the morphism given by left multiplication by 
$U\mu^\sigma(p)\in G(W(R)[1/p])=LG( R)$:
\[
\phi(x)=U\mu^\sigma(p)\cdot x.
\]
After changing the trivialization of $Q$
by right multiplication by $h\in H^\mu(R)$, $U$ changes to 
$U'=h^{-1}U\Phi_{G,\mu}(h)$, and the morphism $\phi$ now is given by multiplication by
$$
U'\mu^\sigma(p)=h^{-1}U\Phi_{G,\mu}(h)\mu^\sigma(p)=
$$
$$=
h^{-1} UF(\mu(p) h\mu(p)^{-1})\mu^\sigma(p)=h^{-1}U\mu^\sigma(p)F(h).
$$
This shows that $\phi: {}^FP[1/p]\xrightarrow{\sim} P[1/p]$ is independent of our choice of trivialization of the $H$-torsor $Q$. We can now easily see using descent that this construction
also produces $\phi: {}^FP[1/p]\xrightarrow{\sim} P[1/p]$ in general.

For $S$ in ${\rm Nilp}_{W(k_0)}$, set $\bar S=S\otimes_{W(k_0)}k_0$ for its special fiber.

\begin{definition}
If $\D$ is a $(G,\mu)$-display over $S$ in ${\rm Nilp}_{W(k_0)}$ we 
set $\D[1/p]=(P[1/p],\phi)$ to be the $G$-isodisplay which is associated to its special
fiber $\D\times_{S}\bar S$, by the above construction.
\end{definition}

  Notice that for $S$ in ${\rm Nilp}_{W(k_0)}$,  we have
\[
(P\times_{L^+G}LG)\times_{S}\bar S=(P\times_{S}\bar S)[1/p].
\]
This is because, if $S=\Spec(R)$, we have $W(R)[1/p]=W(R/pR)[1/p]$, since $p^mR=0$
implies $p^mW(pR)=0$.
 
\begin{definition}
A $G$-quasi-isogeny 
\[
\alpha: \D_1=(P_1, Q_1, u_1)\dashrightarrow \D_2=(P_2, Q_2, u_2)
\]
between two $(G, \mu)$-displays over $S$ in ${\rm Nilp}_{W(k_0)}$ is  an isomorphism
\[
\alpha: (P_1[1/p],\phi_1)\xrightarrow{\sim} (P_2[1/p],\phi_2)
\] between their corresponding $G$-isodisplays,
i.e. an isomorphism of $LG$-torsors
$
\alpha: P_1[1/p]\xrightarrow {\sim} P_2[1/p]
$
which is compatible with $\phi_1$, $\phi_2$, in the sense that 
$ 
\alpha\cdot \phi_1=\phi_2\cdot {}^F\alpha.
$ 
\end{definition}

\begin{remark}\label{glnQiso}
{\rm Recall (Example \ref{glnexample}) that there is an equivalence between the stack of
$(\GL_h, \mu_{d,h})$-displays, and  the stack of Zink (not-necessarily-nilpotent) displays of rank $h$ and 
dimension $d$, over ${\rm Nilp}_{W(k_0)}$. It follows from \cite[Prop. 66]{Zinkdisplay} that, under this equivalence,
  $\GL_n$-quasi-isogenies as defined here, bijectively correspond to quasi-isogenies 
between displays as considered in loc. cit.

  }
\end{remark}
\bigskip

\subsection{The adjoint nilpotent condition}

\subsubsection{}
Suppose that $\D$ is a $(G,\mu)$-display over a $W(k_0)$-scheme $S$.
Let $x$ be a point of $\bar S=S\otimes_{W(k_0)}k_0$ and let ${k(x)}^{\rm ac}$ be an algebraic closure of the residue field $k(x)$.
Set $L(x)=W(k(x)^{\rm ac})[1/p]$. By \ref{displaysAlgCl} applied to $k=k(x)^{\rm ac}$, we see that $\D\times_S k(x)^{\rm ac}$ is described 
by an element $u:=u(x)\in G(W(k(x)^{\rm ac}))$ and that the Frobenius conjugacy class of $b(x)=u(x)\mu^\sigma(p)$ in $G(L(x))$ is well-defined. 
 
\begin{definition}
We will say that the $(G,\mu)$-display $\D$ over $S$ is adjoint nilpotent if, for all $x\in \bar S$,
all the slopes of ${\rm Ad}^G(b(x))\in \GL({\rm Lie}(G))(L(x))$ 
 are $> -1$. We denote by $\B^{\rm nil}(G,\mu)$ the full substack 
 of $\B(G, \mu)$ given by adjoint nilpotent $(G,\mu)$-displays.
\end{definition}

\subsubsection{} We continue with the above assumptions and notations. For every $x\in \bar S$,
consider  the  isocrystal given by the Frobenius semilinear $\phi(x):={\rm Ad}^G(b(x))\circ F_{k(x)^{\rm ac}}$ acting on 
$ \GL({\rm Lie}(G))(L(x))$.
 The Hodge weights 
of this   isocrystal with respect to 
the lattice $\mathcal L_x=\frakg_{W(k(x)^{\rm ac})}={\rm Lie}(G)_{W(k(x)^{\rm ac})}\subset {\rm Lie}(G)_{L(x)}$ are just the weights of $\mu^{-1}$ on $\frakg_{W(k(x)^{\rm ac})}$. Since $\mu$ is 
minuscule these weights lie in the set $\{-1,0,1\}$. Hence, the Newton slopes of the isocrystal are all greater than or equal to $-1$, 
by the Hodge-Newton inequality.
We then obtain that $\D$ is adjoint nilpotent if and only if for all $x\in \bar S$, we have $\phi(x)^{r}\mathcal L_x\subset p^{1-r}\mathcal L_x$, 
with $r={\rm dim}_{\Z_p}(\frakg)$.  
Using this, we see that  we also have (cf. \cite[Definition 3.23]{E7}):

\begin{lemma}
\label{stressful03}
Suppose $\pi$ is the projector on $ W(k_0)\otimes_{\Z_p}{\rm Lie}(G)$ with $\im(\pi)=\Lie U_{\mu^{-1}}$ and 
$\ker(\pi)=\Lie P_\mu$. Let $\D$ be a banal $(G,\mu)$-display over a $k_0$-algebra $R$ and let $U\in G(W(R))$ be a 
representative for it. Then $\D$ is adjoint nilpotent if and only if the endomorphism on $R\otimes_{\z_p}\Lie(G)$ defined by 
$${\rm Ad}^G(w_0(U))\circ(F_R\otimes{\rm id}_{\Lie G})\circ({\rm id}_R\otimes\pi)$$
is nilpotent. 
\end{lemma}
\begin{proof}
When $R=k$ is an algebraically closed field the condition that the endomorphism 
in the statement is nilpotent is equivalent to $$(\phi\cdot p)^r ({\rm Lie}(G)_{W(k)}) \subset p({\rm Lie}(G)_{W(k)}),$$
where $\phi:={\rm Ad}^G(U\mu^{\sigma}(p))\circ F_{k}$; given the above discussion this gives the result in this case. The case of general $R$ follows from this.
\end{proof}

\begin{remark}
\label{stressful01} Suppose that $(G, \mu)=(\GL_h, \mu_{d,h})$. Recall that, by Example \ref{glnexample},
there is an equivalence between 
$\B(\GL_h, \mu_{d,h})$, i.e. the stack of $(\GL_h, \mu_{d, h})$-displays, and the stack of (Zink, not-necessarily-nilpotent) displays of rank $h$ and 
dimension $d$. 

Suppose that $R$
is in ${\rm ANilp}_{W(k_0)}$  and that $\D$ is a  $(\GL_h, \mu_{d, h})$-display
over $R$.
The adjoint nilpotence condition on $\D$ (as defined above) 
is related to Zink's nilpotence condition (\cite[Definition 11/Definition 13]{Zinkdisplay}) on the corresponding Zink display as follows: 

  First assume that $R=k$ is
an  algebraically closed field and that the display $\D$ is given by $b=u\mu_{d,h}(p)\in \GL_h(W(k)[1/p])$: Then we observe that 
 the   isocrystal on ${\rm Lie}(\GL_h)_{W(k)[1/p]}$ given by ${\rm Ad}^G(b)$ has a non-trivial isotypic component of slope $-1$ if and only if the  isocrystal given by $b$ has non-trivial isotypical components for both 
slopes $0$ and $1$. (This is obtained by a simple consideration of roots and weights for $\GL_h$.) For a general $R$ as above, this now implies that $\D$ is adjoint nilpotent
if and only if there exist radical $R$-ideals
$I_{\rm nil}$ and $I_{\rm uni}$ with 
$I_{\rm nil}\cap I_{\rm uni}=\sqrt{pR}$, 
such that both the Zink display $\D_{R/I_{\rm nil}}$ and the dual $(\D_{R/I_{\rm uni}})^t$ satisfy Zink's nilpotence condition. In particular, if either $\D_{R/pR}$ or $(\D_{R/pR})^t$
satisfy Zink's nilpotence condition then $\D$ is adjoint nilpotent. 
\end{remark}

\subsection{Liftings and deformation theory}
\label{deform04} 

In this subsection, we describe the deformation theory of $(G,\mu)$-displays
which satisfy the adjoint nilpotence condition. Many of our results are group theoretic
versions of corresponding results of Zink (\cite{Zinkdisplay}) about deformations of nilpotent displays.

\subsubsection{}

Let $A$ be a $p$-adically separated and complete $W(k_0)$-algebra and fix 
a $p$-adically closed ideal $\fraka\subset A$ with divided powers
which are compatible with the natural divided powers on $pW(k_0)$.
In addition, let us assume that $\fraka$ is $p$-adically topologically nilpotent. Then $A$ is also complete
for the $\fraka$-adic topology. (The most useful case is when $A$ is in ${\rm ANilp}_{W(k_0)}$
and $\fraka\subset A$ is a nilpotent pd-ideal.)

We set 
$I_\fraka(A):=W(\fraka)+I(A)$, which is a pd-ideal of $W(A)$.
In this situation, Zink's logarithmic ghost coordinates 
(\cite[\S 1.4]{Zinkdisplay}) establish an 
isomorphism 
\[
\log: W(\fraka)\xrightarrow{\simeq}\prod_{i=0}^\infty\fraka
\]
 leading to 
splittings $W(\fraka)=\fraka\oplus I(\fraka)$ and 
$I_\fraka(A)=\fraka\oplus I(A)$. We will write the elements of 
$\prod_{i=0}^\infty\fraka$ using brackets $[a_0,a_1,\ldots ]$. 
We have $F([a_0, a_1, \ldots ])=[pa_1, pa_2, \ldots]$, so $F(\fraka)=0$.
The splittings allow one to 
define the important $F$-linear extension $V_\fraka^{-1}:
I_\fraka(A)\rightarrow W(A)$ of $V^{-1}=V_\fraka^{-1}|_{I(A)}$ by setting 
$0=V_\fraka^{-1}|_{\fraka}$ (see \cite[Lemma 38]{Zinkdisplay}). 

\subsubsection{} Now let  $(G,\mu)$ be as above. We set
$  H^\mu(A,\fraka)$ to be the inverse image of $H^\mu(A/\fraka)$ in 
$L^+G(A)$. Using an argument as in the proof of Proposition \ref{productI}
we deduce that 
\[
  H^\mu(A,\fraka)=U_{\mu^{-1}}(W(\fraka))H^\mu(A).
\]
We obtain further decompositions 
\[
 H^\mu(A,\fraka)=U_{\mu^{-1}}(\fraka)H^\mu(A)=G(\fraka)H^\mu(A), 
\] 
where $G(\fraka) ={\rm ker}(G(A)\to G(A/\fraka))$ and similarly for $U_{\mu^{-1}}(\fraka)$. We can now see that  $\Phi_{G,\mu}(A):H^\mu(A)\rightarrow L^+G(A)$   can be extended to a map 
\[
\Psi_\fraka:  H^\mu(A,\fraka)\rightarrow L^+G(A)
\]
vanishing on $G(\fraka)$. (Notice that $\Phi_{G,\mu}$ takes the identity value on $P_\mu(\fraka)=H^\mu(A)\cap G(\fraka)$, cf. \cite[Cor. 3.11]{E7}.) We similarly get 
\[
\Psi_\fraka(G(W(\fraka))\subset G(W(\fraka)).
\]

 \subsubsection{} Recall the weight decomposition 
\[
W(k_0)\otimes_{\z_p}\frakg=\frakp\oplus\fraku^-
\] 
 induced by our minuscule 
cocharacter $\mu$.  
We can also construct Lie-theoretic analogs of the group $H^\mu(A,\fraka)$ and the homomorphism $\Psi_\fraka$. Namely, we set
\begin{equation}
 \frakh_\fraka^\mu(A)=(I_\fraka(A)\otimes_{W(k_0)}\fraku^-)\oplus (W(A)\otimes_{W(k_0)}\frakp),
\end{equation}
and we define 
\begin{equation}
\label{blabla}
\psi_\fraka: \frakh_\fraka^\mu(A)\rightarrow W(A)\otimes_{\z_p}\frakg
\end{equation}
 to be $V_\fraka^{-1}$ on $I_\fraka(A)\otimes_{W(k_0)}\fraku^-$, and   $p^mF$ 
on all summands of $\mu$-weight $m\geq0$.

\begin{theorem}
\label{GMZCF}
Under the above assumptions on $\fraka\subset A$, and with the above notations, suppose in addition
that we have $U$, $U'\in L^+G(A)=G(W(A))$ that satisfy the adjoint nilpotence condition and are such that
\[
U'\equiv U\ {\rm mod}\ W(\fraka).
 \] 
Then there is a unique 
element $h\in G(W(\fraka))$ such that $
U'=h^{-1}U\Psi_\fraka(h)$.
\end{theorem}
\begin{proof}
We endow $G(W(\fraka))={\rm ker}(G(W(A))\to G(W(A/\fraka)))$ with the topology coming from 
the restriction of the $p$-adic topology of $A$, with respect to which 
it is separated and complete. Consider the chain of ideals
$\fraka_n=\fraka(pA+\fraka)^n$ which are pd-subideals of $\fraka$. Then the ideal $\fraka$ is complete and separated 
for the topology given by $(\fraka_n)$.  
Note that $\Psi_\fraka(G(W(\fraka_n))\subset G(W(\fraka_n))$, 
because $\fraka_n$ is a pd-subideal 
of $\fraka$. By an inductive procedure we will construct elements 
$h_n\in G(W(\fraka))$ such that
$$
h_n\equiv h_{n-1}\ {\rm mod}\  G(W(\fraka_{n-1}))
$$
and
$$
h_n^{-1}U\Psi_\fraka(h_n)\equiv U'\ {\rm mod}\  G(W(\fraka_n)):
$$ 
Then $h=\varprojlim_n h_n$ is the required element.
Set $h_0=1$ and for $n\geq 1$ consider 
the element 
$$
U'':=h_{n-1}^{-1}U\Psi_\fraka(h_{n-1}).
$$
By the induction hypothesis we have
$$
U''\equiv U'\ {\rm mod}\  G(W(\fraka_{n-1})).
$$ 
We can define a function $K$  
from 
$$
W(\fraka_{n-1}/\fraka_n)\otimes_{\z_p}\frakg\  \buildrel{\rm exp}\over {\cong}\ G(W(\fraka_{n-1}/\fraka_n))\cong G(W(\fraka_{n-1}))/G(W(\fraka_n)).
$$ 
to itself by setting
$$
K(X):={\rm exp}^{-1}(U''\Psi_\fraka(\widetilde{{\rm exp}(X)}){U'}^{-1}\ {\rm mod}\ G(W(\fraka_{n}))),
$$
where the tilde denotes a lift of ${\rm exp}(X)\in G(W(\fraka_{n-1}/\fraka_n))$ to $G(W(\fraka_{n-1}))$.
For $ Y, X\in W(\fraka_{n-1}/\fraka_n)\otimes_{\z_p}\frakg$, we have
\begin{equation}\label{addK}
K(Y+X)=K(Y)+(\Ad^G(U)\circ\psi_{\fraka_{n-1}/\fraka_n})(X).
\end{equation}
(Here $\psi_{\fraka_{n-1}/\fraka_n}$ is the map of our construction above applied to
the ring $A/\fraka_n$ and its pd ideal $\fraka_{n-1}/\fraka_n$. Note also that since $U$, $U'$ and $U''$ are 
congruent modulo  $W(\fraka_{n-1})$, the maps $\Ad^G(U)$, $\Ad^G(U')$ and $\Ad^G(U'')$
induce the same operator on $W(\fraka_{n-1}/\fraka_n)\otimes_{\Z_p}\frakg$.)

Now observe that $W(\fraka_{n-1}/\fraka_n)\simeq \prod_{i=0}^\infty \fraka_{n-1}/\fraka_n$ can be regarded as a $A/(pA+\fraka)$-module; the module structure on the $i$-th component of the right hand side
is given by $s\cdot a=F^i(s) a$, where $F$ is the absolute Frobenius on $A/(pA+\fraka)$. 
Let
$$
\psi^*:(A/(pA+\fraka))\otimes_{\z_p}\frakg\rightarrow(A/(pA+\fraka))\otimes_{\z_p}\frakg
$$
be the composition of the absolute Frobenius on 
$(A/(pA+\fraka))\otimes_{\z_p}\frakg$ with the projection $\pi$ of 
$W(k_0)\otimes_{\z_p}\frakg$ onto $\fraku^-$ killing $\frakp$. 
Since $V^{-1}_{\fraka}([x_0, x_1, x_2, \ldots ])=[x_1, x_2, \ldots]$, we can 
deduce from the definition of $\psi_{\fraka_{n-1}/\fraka_n}$ that
\[
\psi_{\fraka_{n-1}/\fraka_n}= \id_{W(\fraka_{n-1}/\fraka_n)}\otimes_{A/(pA+\fraka)}\psi^*.
\]
This together with Lemma  \ref{stressful03} shows  that the adjoint nilpotence of $U$ implies that $\Ad^G(U)\circ\psi_{\fraka_{n-1}/\fraka_n}$ is nilpotent on $W(\fraka_{n-1}/\fraka_n)\otimes_{\Z_p}\frakg$. In turn, this nilpotence together with (\ref{addK}) implies 
that $K$ has a unique fixed point in $G(W(\fraka_{n-1}/\fraka_n))\simeq W(\fraka_{n-1}/\fraka_n)\otimes_{\Z_p}\frakg$. Finally, let us choose  a lift $h_*$ to $G(W(\fraka_{n-1}))$ of the fixed point 
of $K$, and let us write $h_n:=h_{n-1}h_*$. We can now see
that 
\[
h_n^{-1}U\Psi_\fraka(h_n)=h^{-1}_*h_{n-1}^{-1}U\Psi_\fraka(h_{n-1})\Psi_\fraka(h_*)=h^{-1}_*U''\Psi_\fraka(h_*)=U'
\]
modulo $G(W(\fraka_{n}))$. This completes the inductive step that shows the existence of $h$. (Notice that this is a group theoretic generalization
of the proof of \cite[Theorem 44]{Zinkdisplay}.)

Now let us show the uniqueness of $h$: We can see that it is enough
to show that $h^{-1}U\Psi_\fraka(h)=U=U'$ and $h\in G(W(\fraka))$ imply $h=1$.
In fact, it is enough by induction to assume that $h\in G(W(\fraka_{n-1}))$.
Apply the argument above  to $h$ instead of $h_{n-1}$; then we have $U''\equiv U\ {\rm mod}\ G(W(\fraka_{n-1}))$. Since obviously $K(0)=0$,
the  uniqueness of the fixed point of the map $K$ implies that $h\in G(W(\fraka_n))$.
Since $\cap_n \fraka_n=(0)$, we conclude that $h=1$.
\end{proof}

\subsubsection{} Recall that by Lemma \ref{diagonal} the diagonal of $\B(G, \mu)$ is affine. We will now show 
that the diagonal of $\B^{\rm nil}(G, \mu)$ is $p$-adically formally 
unramified, cf. \cite[Prop. 40]{Zinkdisplay}, \cite[Cor. 3.26(ii)]{E7}:

\begin{corollary}
\label{deform06}
Let $A$ be a $\fraka$-adically separated algebra in ${\rm ANilp}_{W(k_0)}$, where $\fraka$ is 
an ideal that contains some power of $p$. Let $\phi$ be an automorphism 
of an adjoint nilpotent $(G,\mu)$-display $\D$ over $\Spec A$. Then 
$\phi_{A/\fraka}$ is the identity if and only if $\phi$ is the identity.
\end{corollary}
\begin{proof}
It is enough to prove, that if $\phi_{A/\fraka}$ is the identity, then $\phi_{A/\fraka^n}$ is the identity for 
every $n$, since by Lemma \ref{diagonal} (b), the diagonal  of the stack $(G,\mu)$-displays is
affine, hence separated. By a straightforward induction argument it suffices to do this under the additional assumption that $\fraka^2=0$. After passing to some affine \'etale covering we can also assume that $\D$ is banal. Then since  ideals of vanishing square have a natural pd-structure
we can apply Theorem \ref{GMZCF} and the result follows immediately from the uniqueness in the statement there. 
\end{proof}

\subsubsection{} Suppose that $\D_0$ is a adjoint nilpotent   $(G,\mu)$-display over $A_0=A/\fraka$ which is banal, and so given by 
a $U_0\in G(W(A_0))$. Using Theorem \ref{GMZCF} we will classify lifts (``deformations'') of $\D_0$ to displays over  $A$, up to isomorphism. 
By definition, a lift of $\D_0$ is a pair $(\D, \delta)$ of a $(G, \mu)$ display $\D$ over $A$ together with an isomorphism $\delta: \D_0\xrightarrow{\sim} \D\otimes_AA/\fraka$. 
Since $\D$ is also banal by \ref{torsCor}, such pairs are given by pairs $(U, h_0)$ with $U\in G(W(A))$, $h_0\in H^\mu(A_0)$, such that 
\[
U\,{\rm mod}\,\fraka=h^{-1}_0U_0\Phi_{G,\mu}(h_0);
\]
two such pairs $(U, h_0)$, $(U', h_0')$, are isomorphic if there is $h\in H^\mu(A)$ with 
\[
h\, {\rm mod}\, \fraka=h^{-1}_0 h'_0, \qquad U'=h^{-1}U\Phi_{G,\mu}(h).
\]
Suppose we are given a pair $(U, h_0)$ up to isomorphism. 
Under our assumptions, since also $H^\mu$ is formally smooth, we can lift $h_0\in H^\mu(A_0)$ to $\tilde h_0\in H^\mu(A)$;
this gives an isomorphism of $(U, h_0)$ to a pair of the form $(U', 1)$, i.e. with $U'\,{\rm mod}\, \fraka=U_0$;
hence, in classifying pairs $(U, h_0)$ up to isomorphism, we can always assume that the second component
is trivial, i.e. $h_0=1$; then $U\,{\rm mod}\, \fraka=U_0$. Suppose that $(U,1)$, $(U', 1)$, are two such pairs.
By Theorem \ref{GMZCF} there is unique $g\in G(W(\fraka))$ such that 
\[
U'=g^{-1}U\Psi_\fraka (g).
\]
Using this and the fact that the restriction of $\Psi_\fraka$ to $H^\mu(\fraka)$ coincides with $\Phi_{G, \mu}$,
we see that the following is true: 

After choosing a lift $U$ of $U_0$,
we can identify the set of lifts of the $(G,\mu)$-display $\D_0$ up to isomorphism with the set of right cosets $G(W(\fraka))/H^\mu(\fraka)$;
the bijection is given by
\[
gH^\mu(\fraka)\mapsto [(g^{-1}U\Psi_\fraka(g), 1)].
\]
We can see that the natural map 
\[
G(W(\fraka))/H^\mu(\fraka)\hookrightarrow G(W(A))/H^\mu(A)\to (G/P_\mu)(A)
\]
identifies $G(W(\fraka))/H^\mu(\fraka)$ with the set of $A$-valued points of the homogeneous space $G/P_\mu$
that reduce to the identity coset modulo $\fraka$. (As usual, we can then also see that this set is in bijection with
the set of liftings of the Hodge filtration of $\D_0$.)

\subsubsection{}\label{affine deformation} Continue with the above assumptions and notations but 
suppose that in addition we have $\fraka^2=(0)$.
Then we have as above
\[
\fraka\otimes_{W(k_0)} \fraku^- \xrightarrow{\sim} G(W(\fraka))/H^\mu(\fraka).
\]
Hence, if $\fraka^2=(0)$, the choice of a lift $U$ of $U_0$ gives a bijection 
between the set of all lifts of $\D_0$ up to isomorphism and the set $\fraka\otimes_{W(k_0)} \fraku^-$. In fact, 
the above discussion shows that there is an action of the group $\fraka\otimes_{W(k_0)}\fraku^-$ on the 
set of all lifts of $\D_0$ up to isomorphism, and this set is a (trivial) principal homogeneous space for
the group  $\fraka\otimes_{W(k_0)}\fraku^-$.

\subsubsection{} Suppose that $k$ is a separably closed field 
of characteristic $p$ and let $\D_0$ be an adjoint nilpotent $(G,\mu)$ display over $S_0=\Spec(k)$.
We can then consider the functor ${\rm Def}_{\D_0}$ of formal deformations of $\D_0$.
This is a functor on the category of augmented local Artinian $W(k)$-algebras, i.e. local Artinian $W(k)$-algebras $(A, \frakm)$ with an 
isomorrphism $A/\frakm\xrightarrow{\sim} k$ with ${\rm Def}_{\D_0}(A)$ the set of isomorphism classes of lifts of $\D_0$
to a $(G, \mu)$-display over $A$. Here, as above, a lift of $\D_0$ is by definition a pair $(\D, \delta)$ of a $(G,\mu)$-display $\D$ over $A$ together with an isomorphism  $\delta: \D_0\xrightarrow{\sim} \D\otimes_A k$.

The display $\D_0$ is  banal and given by a matrix $U_0\in G(W(k))$. Choose a basis $e_1,\ldots , e_r$
of the $W(k_0)$-module $\fraku^-=\Lie(U_{\mu^{-1}})$. By \ref{minusculeunipotent} this induces 
an isomorphism ${\mathbb G}_a^r\xrightarrow{\sim} U_{\mu^{-1}}$; we will write by ${\rm exp}(a_1e_1+\cdots +a_re_r)$
the point of $U_{\mu^{-1}}$ which is the image of $(a_1,\ldots, a_r)$. 
Set $A=W(k)[[t_1,\ldots , t_r]]$ for the power series ring with $r$ variables.
(More canonically, we can take $A$ to be the formal completion of $U_{\mu^{-1}}\otimes_{W(k_0)}W(k)$ at the origin.)
Let us set
\[
g_{\rm uni}={\rm exp}([t_1]e_1+\cdots +[t_r]e_r)\in U_{\mu^{-1}}(W(A))\subset G(W(A))
\]
where $[t_i]=(t_i, 0, 0, \ldots )\in W(A)$ is the Teichmuller lift.
Set
\[
U_{\rm uni}=g^{-1}_{\rm uni} U_0\in G(W(A)).
\]
The element $U_{\rm uni}$ defines a $(G, \mu)$-display $\D_{\rm uni}$ over $A$. 

We claim that
$\D_{\rm uni}$ prorepresents the functor ${\rm Def}_{\D_0}$ of formal deformations of $\D_0$:

Given the above work,
the proof of this is very similar to the proof of the corresponding statement for Zink displays 
given in \cite[p. 173-176]{Zinkdisplay}. We will just sketch the argument here: We first observe that, by the discussion
in the above paragraph, ${\rm Def}_{\D_0}$ is formally smooth. Next, we notice that Theorem \ref{GMZCF} (or the discussion above) quickly implies 
that the reduction of $\D_{\rm uni}$ over $A_2:=A/(t_1, \ldots , t_r)^2$ is  universal for deformations over augmented local Artinian $W(k)$-algebras with  maximal ideal of square zero. 
(Notice that over $A_2$, we have  $\Psi_{(t_1,\ldots, t_r)}(g_{\rm uni})=1$.) This implies  that the morphism of functors on augmented local Artinian $W(k)$-algebras 
$\epsilon: {\rm Spf}(A)\to {\rm Def}_{\D_0}$ given by $\D_{\rm uni}$ induces an isomorphism on tangent spaces. It now follows that $\epsilon$ is an isomorphism and this
concludes the proof. The same argument works even if the field $k$ is not separably closed provided $\D_0$ is banal.

\subsubsection{}
We can similarly obtain a global version of some deformation theory statements in which we
consider adjoint nilpotent $(G,\mu)$-displays over general schemes in  ${\rm Nilp}_{W(k_0)}$
(cf. \cite[subsection 3.5]{E7} and especially \cite[corollary 3.27]{E7})). 

 Let $S$ be a scheme in ${\rm Nilp}_{W(k_0)}$ 
and recall that if $\D=(P, Q, u)$ is a $(G,\mu)$-display over $S$, then the Hodge filtration of $\D$ is a  section 
\[
\alpha: S\to P/H^\mu
\]
of the $X_\mu=G/P_\mu$-bundle $P/H^\mu\to S$.
Recall also the vector bundle 
$
T_\D
$
 over $S$ obtained by pulling
back via $\alpha$ the normal bundle of the (regular) closed immersion
$\alpha(S)\hookrightarrow \P/H^\mu$.

Suppose $\mathcal I\subset\O_S$ is an ideal sheaf with $\mathcal I^2=0$. For any 
$S$-scheme $X$ we denote $S_0\times_SX=X_0$, where $S_0\hookrightarrow S$ 
is the nilimmersion defined by $\mathcal I$. We can view 
$$
0\rightarrow\mathcal I\hookrightarrow\O_S\twoheadrightarrow\O_{S_0}\rightarrow0
$$ 
as an exact sequence on abelian sheaves on the fpqc site of $S$: More specifically let 
$\O_{S_0}(X)$, $\O_S(X)$, and $\mathcal I(X)$ be the global sections of $X_0$, 
$X$, and the kernel of $\O_S(X)\rightarrow\O_{S_0}(X)$.
 
 Now let $\D_0$ be an adjoint nilpotent $(G,\mu)$-display 
over $S_0$.  Again, by a lift of $\D_0$ over an $S$-scheme $X$ we mean a pair $(\D,\delta)$ where 
$\D$ is a $(G,\mu)$-display over $X$, and 
$\delta:\D_0\times_{S_0}X_0\xrightarrow{\sim}\D\times_XX_0$ is an isomorphism.
By  Corollary \ref{deform06} no lift has any automorphism other 
than the identity. Let ${\rm D}_{\D_0,S}(X)$ be the set of isomorphism classes of lifts over $X$. 
By pull-back of lifts this is a presheaf on $S$. Let us check that ${\rm D}_{\D_0,S}$ satisfies the sheaf axiom for a faithfully flat 
map $Y\rightarrow X$: If the pull-backs $\pr_1^*(\D)$ and 
$\pr_2^*(\D)$ of some $(\D,\delta)\in{\rm D}_{\D_0,S}(Y)$ agree for the two projections $\pr_1,\,\pr_2:Y\times_XY\rightarrow Y$, then this means that there 
exists $\beta:\pr_1^*(\D)\stackrel{\simeq}{\rightarrow}\pr_2^*(\D)$ which restricts to the identity on $Y_0\times_{X_0}Y_0$. We easily deduce the cocycle 
condition for $\beta$, because any equality of isomorphisms of displays over $Y\times_XY\times_XY$  can be checked over $Y_0\times_{X_0}Y_0\times_{X_0}Y_0$, 
by Corollary \ref{deform06}. This shows that $\D$ (resp. $\delta$) descends to $X$ (resp. $X_0$). It follows that ${\rm D}_{\D_0,S}$ is a fpqc-sheaf on $S$.

\begin{theorem}\label{globaldef}
Suppose that $\D_0$ is  an adjoint nilpotent $(G,\mu)$-display 
over $S_0$ and that $S_0\hookrightarrow S$ is a closed immersion defined by an ideal sheaf $\mathcal I$ of square zero. 
Then the above functor ${\rm D}_{\D_0,S}$ has the structure of a locally trivial principal homogeneous space for the sheaf 
$\mathcal I\otimes_{\O_{S}}T_{\D_0}$. In particular, the set ${\rm D}_{\D_0,S}(S)$ of isomorphism classes of lifts of $\D_0$ over $S$ is either empty or is a principal homogeneous space for $\Gamma(S, \mathcal I\otimes_{\O_{S}}T_{\D_0})$.
\end{theorem}

\begin{proof}
The case that $S$ is affine and the display is banal is given by \ref{affine deformation}.
(To see this one uses the fact that the Lie algebra $\fraku^-$ gives the tangent space
of the homogeneous space $X_\mu$ at the identity 
$1\cdot P_\mu$.) The general case follows from this and descent.
\end{proof}

\begin{corollary}
\label{globtriv}
If $S_0$ is affine, then ${\rm D}_{\D_0,S}$ is globally trivial, i.e. $
{\rm D}_{\D_0,S}(S)\neq\emptyset$. \ $\square$
\end{corollary}

\subsection{Faithfulness up to isogeny}

If $B$ is a ring such that both $pB$ and $\sqrt{0_B}$ are nilpotent the 
forgetful functor from (Zink-)displays   to isodisplays is faithful (see \cite[p. 186]{Zinkdisplay}). 
Here, we describe an extension of this which needs the notion of Frobenius separatedness defined in Appendix C. We write   ${\rm ANilp}_{W}^{\rm aFs}$ for the full 
subcategory of ${\rm ANilp}_W$ consisting of $W$-algebras which are almost Frobenius separated.

\begin{proposition}
\label{GNN01}
Suppose that $B$ is in ${\rm ANilp}^{\rm aFs}_W$. Then the 
functor of \ref{Gisodisplay} from adjoint nilpotent $(G,\mu)$--displays 
over $B$ to $G$-isodisplays over $B/pB$ is faithful.
\end{proposition}
\begin{proof}
Let $\phi$ be an automorphism of a display $\D=(P,Q, u)$ with $( G,\mu)$-structure 
over $B$, such that $\phi$ and $\id_\D$ give rise to the same self-$G$-quasi-isogeny. We would like to prove $\phi=\id_\D$. Using 
Corollary \ref{deform06} and since $p$ is nilpotent in $B$, we can easily reduce to the case $pB=0$. Let 
$\fraka$ be the smallest ideal of $B$ such that $\phi\equiv\id_\D\ {\rm mod}\ \fraka$;
this ideal exists since by Lemma \ref{diagonal} (b) the diagonal of $\B(G,\mu)$
is representable and affine. In fact, the $B$-scheme representing the automorphisms of the display
$\D$ is a closed subscheme of the Greenberg transform $F{\mathcal I}$ of $\mathcal I:=\underline{\rm Aut}_G(\P)$,
where $\P$ is the $G$-torsor over $W(B)$ that corresponds to the $L^+G$-torsor $P$ by Proposition \ref{torsor02}.
Corollary \ref{deform06} on the rigidity of automorphisms  implies  that $\fraka=\fraka^2$. 
 For simplicity, set $J_n=\ker(F^n_{B})=\{x\in B |x^{p^n}=0\}$. 
Note that $W(J_n)=W(B)[p^n]$, and
$$
\bigcup\nolimits_n W(J_n)=\ker(W(B)\rightarrow W(B)_\q=W(B)[1/p]).
$$
Since $G$ is of finite type over $\Z_p$, by descent we see that the $G$-torsor
$\P$ and also the scheme $\mathcal I$, are of finite presentation
over $W(B)$. Since $\phi$ becomes the identity in $F_{(p)}{\mathcal I}(B)={\mathcal I}( W(B)[1/p])$, 
there must exist some integer $N$ such that $\phi$ is equal to the identity in 
${\mathcal I}(W(B/J_N))$. We deduce that $\fraka \subset  J_N$ which is 
bounded nilpotent, while $\fraka=\fraka^2$.  Hence, by 
Lemma \ref{FsBNI}, $\fraka=0$ which completes the proof.\end{proof}

\subsection{Display blocks}\label{blockpar}

In this subsection, we define and discuss the notion of a display block. This is a technical construction 
which is useful for handling a certain type of descent that appears in the proof of Proposition \ref{analog} and of  the main result 
Theorem \ref{hodgemain}. The subsection can be omitted at
first reading.

\subsubsection{}\label{embeddingpair}
We consider two pairs $(G, \mu)$ and $(G',\mu')$ where both $G$ and $G'$ are reductive group schemes over $\z_p$, 
and $\mu$ and $\mu'$ are, respectively, minuscule cocharacters of $G$ and $G'$ defined over $W(k_0)$. We suppose 
that there is a group scheme homomorphism $i: G\to G'$ which is a closed immersion and is such that 
$\mu'=i\cdot\mu$ as in \ref{twopairs}. We will denote this set-up by writing $i:(G,\mu)\hookrightarrow(G',\mu')$. 

Note that if for a $(G, \mu)$-display 
$\D$ the $(G', \mu')$-display $i(\D)$ is adjoint nilpotent then $\D$ is 
also adjoint nilpotent (since the Frobenius isocrystal given by ${\rm Ad}^G(\D)[1/p]$
is a sub-isocrystal of the one given by ${\rm Ad}^{G'}(i(\D))[1/p]$). 

\begin{definition}
\label{block}
Fix $i:(G,\mu)\hookrightarrow(G',\mu')$ as above.
Consider an injective homomorphism $A\hookrightarrow B$
in ${\rm ANilp}_{W(k_0)}$. An
 $(A,B)$-display block is a triple $(\D',\D,\psi)$ where
\begin{itemize}
\item
$\D'$ is a $(G', \mu')$-display over $A$,
\item
$\D$ is a $(G, \mu)$-display over $B$, and
\item
$\psi: i(\D)\xrightarrow{\sim} \D'\times_AB$ is an isomorphism of $(G',\mu')$-displays.
\end{itemize}
An isomorphism of two $(A,B)$-display blocks 
\[
(\C',\C,\phi)\xrightarrow{\sim} (\D',\D,\psi).
\]
is a pair of isomorphisms 
$\epsilon':\C'\xrightarrow{\sim}\D'$ and $\epsilon:\C\xrightarrow{\sim}\D$ satisfying 
\[
\psi\circ i(\epsilon)\circ\phi^{-1}=\epsilon'\times_AB.
\]
\end{definition}

\begin{definition}
The $(A, B)$-display block
$(\D',\D,\psi)$ is called {\sl effective} if there exists 
a $(G, \mu)$-display $\D^\Diamond$ over $A$
such that 
\[
(\D',\D,\psi)\cong (i(\D^\Diamond),\D^\Diamond\times_AB,i(\id_{\D^\Diamond})\times_AB).
\] 
Then the $(G, \mu)$-display $\D^\Diamond$ is unique up to a unique isomorphism.
\end{definition}

\subsubsection{} Consider an ideal $\frakb\subset B$ and write $\fraka:=\frakb\cap A\subset A$. By the 
reduction of $(\D',\D,\psi)$ modulo $\frakb$ we mean the $(A/\fraka,B/\frakb)$-display block given by 
the triple $(\D'\times_AA/\fraka,\D\times_BB/\frakb,\psi\times_BB/\frakb)$.

\subsubsection{}  
We first need the following result, 
which is a slight generalization of \cite[Corollary 3.28]{E7}:
\begin{lemma}
\label{deform02}
Suppose $i:(G,\mu)\hookrightarrow(G',\mu')$ is as above, and consider 
 homomorphisms $A\hookrightarrow B\twoheadrightarrow B/\frakb$
in ${\rm ANilp}_{W(k_0)}$. We assume that
\begin{itemize}
\item[(i)]
$B$ is $\frakb$-adically separated, i.e. $(0)=\bigcap_n\frakb^n$, and
\item[(ii)]
$A$ is $\frakb$-adically closed as a subset of $B$, i.e. $A=\bigcap_n(A+\frakb^n)$.
\end{itemize}
Let $\fraka:=A\cap\frakb$ so that we have
 $A/\fraka\hookrightarrow B/\frakb$. 

Suppose 
that $\D_1$ and $\D_2$ are two  $(G,\mu)$-displays over $A$. We assume that 
$i(\D_j\times_AA/\fraka)$, $j=1$, $2$ are adjoint-nilpotent $(G',\mu')$-displays 
(so that then  $\D_j\times_AA/\fraka$ are also adjoint-nilpotent 
$(G,\mu)$-displays). Let  
$$
\psi: i(\D_1)\times_AB\xrightarrow{\sim}i(D_2)\times_AB
$$ 
be a $(G',\mu')$-isomorphism over $B$, and 
$$
\phi_0: \D_1\times_AA/\fraka\xrightarrow{\sim}\D_2\times_AA/\fraka
$$
a $(G,\mu)$-isomorphism over $A/\fraka$. Assume the base changes of $\phi_0$ 
and $\psi$ to $B/\frakb$ are compatible in the obvious sense. Then there is a 
unique $(G, \mu)$-isomorphism  $\phi: \D_1\xrightarrow{\sim}\D_2$ over $A$
that induces both $\phi_0$ and $\psi$.
\end{lemma}
\begin{proof} 
We first deal with the case that $\frakb$ is nilpotent. Then, by induction 
we can assume that $\frakb^2=0$. Now we can use the map $\phi_0$ in order 
to view $\D_1$ as a lift of $\bar\D_2:=\D_2\times_AA/\fraka$, so that there exists a well-defined element $N\in T_{\bar\D_2}\otimes_{A}\fraka$ that measures the difference between the elements $\D_1$ and $\D_2$ of the 
deformation ${\rm D}_{\bar\D_2,\Spec A}$. The existence of $\psi$ implies that the image of $N$ in
$T_{i(\bar\D_2)}\otimes_{A}\frakb$ has to vanish. Since
$T_{\bar\D_2}\otimes_{A}\fraka\rightarrow 
T_{i(\bar\D_2)}\otimes_{A}\frakb$ is injective, this yields $N=0$ which implies the existence of $\phi$ over $A$. 

We now deal with the general case: The argument above gives that, for every $n\geq 1$, there is a $(G, \mu)$-display isomorphism
$\phi_n:\D_1\times_AA/(A\cap\frakb^n)
\xrightarrow{\sim}\D_2\times_AA/(A\cap \frakb^n)$ which lifts $\phi_0$ 
and is compatible with $\psi\times_B B/\frakb^n$. This, since the diagonal of $\B(G',\mu')$ is affine
by Lemma \ref{diagonal}, implies that $\psi$, as a $(G',\mu')$-display isomorphism, is actually defined over the subring 
$A+\frakb^n\subset B$ and therefore over $A=\bigcap_n (A+\frakb^n)$. Now, 
by the above, $\psi\times_A A/(A\cap \frakb^n)=\phi_n$ which 
is a $(G, \mu)$-isomorphism for all $n\geq 1$. Since 
$\bigcap_n (A\cap\frakb^n)=0$ we obtain that $\psi$ is also 
a $(G,\mu)$-isomorphism.
\end{proof}

\subsubsection{} 
We end this paragraph with two lemmas that,
roughly speaking, show that under some assumptions, certain deformations and liftings of effective display blocks are still effective. In both of these, we fix
$i:(G,\mu)\hookrightarrow(G',\mu')$ as above.

\begin{lemma}
\label{deform05}
Consider  homomorphisms $A\hookrightarrow B\twoheadrightarrow B/\frakb$ 
in ${\rm ANilp}_{W(k_0)}$. Suppose that   
$\frakb$ is nilpotent and set $\fraka:=\frakb\cap A$. Let $\D_0^\Diamond $ be a $(G, \mu)$-display over $A/\fraka$ such that $i(\D_0^\Diamond)$ is an adjoint nilpotent 
$(G',\mu')$-display. Then the assignment 
\[
\D^\Diamond\mapsto(i(\D^\Diamond),\D^\Diamond\times_AB,i(\id_{\D^\Diamond})\times_AB)
\]
 induces a canonical bijection 
from the set of isomorphism classes of deformations of the $(G,\mu)$-display $\D_0^\Diamond$ over $A$ to the set of isomorphism classes of deformations of the 
$(A/\fraka,B/\frakb)$-display block 
\[
(i(\D_0^\Diamond),\D_0^\Diamond\times_{A/\fraka}B/\frakb,i(\id_{\D_0^\Diamond})\times_{A/\fraka}B/\frakb)
\] over $(A,B)$.
\end{lemma}  
\begin{proof}
We can assume $\frakb^2=0$, by induction. In this case, Corollary \ref{globtriv} shows that both of the two sets in question are principal homogeneous spaces, 
the former one under $T_{\D_0^\Diamond}\otimes_{A}\fraka$ and the latter one under 
\[
(T_{i(\D_0^\Diamond)}\otimes_{A}\fraka)\cap (T_{\D_0^\Diamond}\otimes_{A}\frakb),
\] the intersection taking place in 
$T_{i(\D_0^\Diamond)}\otimes_{A}\frakb$. We get the result as these two groups are equal.
\end{proof}

\begin{lemma}
\label{deform01}
Consider  homomorphisms $A\hookrightarrow B\twoheadrightarrow B/\frakb$ 
in ${\rm ANilp}_{W(k_0)}$. Assume that $(A, \frakm)$ 
is a complete local Noetherian ring, and that one of the following two assertions holds:
\begin{itemize}
\item[(i)]
There exists a regular sequence $(f, g)$ in $\frakm$ such that $B=A[g^{-1}]$ and $\frakb$ is generated by $f$.
\item[(ii)]
$B$ is a finitely generated $A$-module and $\frakb=J(B)$.
\end{itemize}
Let $(\D',\D,\psi)$ be a $(A,B)$-display block with $\D'$ adjoint nilpotent whose
 reduction modulo $\frakb$ is effective and given by a  $(G,\mu)$-display over $A/A\cap \frakb$ which is banal 
 over $A/\frakm$.  Then $(\D',\D,\psi)$ is effective.
\end{lemma}
\begin{proof}
Let us write $\fraka_n$ for the intersection $A\cap\frakb^n$. In case (ii) the Artin-Rees lemma implies that 
the ideals $\fraka_n$ form a basis of neighborhoods for the $\frakm$-adic topology on $A$.
In case (i) each $\fraka_n$ 
is generated by $f^n$. Thus, in both cases $B$ is $\frakb$-adically separated and $A$ is $\frakb$-adically closed as a subset of $B$, as 
$\fraka_n$ tends $\frakm$-adically towards $0_A$. By Lemma \ref{deform05} we can choose a sequence of suitable adjoint nilpotent $(G,\mu)$-displays $\D_n^\Diamond$ over $A/\fraka_n$ and isomorphisms:
$$
(i(\D_n^\Diamond),\D_n^\Diamond\times_{A/\fraka_n}B/\frakb^n,i(\id_{\D_n^\Diamond})\times_{A/\fraka_n}B/\frakb^n)
\cong(\D',\D,\psi)\ \mathrm{mod\ } \frakb^n.
$$ 
Using the existence of the universal deformation 
over the formal deformation space of the (banal) $(G,\mu)$-display $\D_1^\Diamond\times_{A/A\cap\frakb}A/\frakm$ over $A/\frakm$ 
we can now construct a ``limit'' $(G,\mu)$-display 
$\varprojlim_n\D_n^\Diamond=\D^\Diamond$ over $A=\varprojlim_n A/\fraka_n$
(cf. Proposition \ref{limitdisplay}). Passing to the limit we get:
$$
(i(\D^\Diamond),\D^\Diamond\times_A\hat B,i(\id_{\D^\Diamond})\times_A\hat B)
\cong(\D',\D\times_B\hat B,\psi\times_B\hat B),
$$
where $\hat B=\varprojlim_n B/\frakb^n$. Since $B$ is $\frakb$-adically separated $B\hookrightarrow\hat B$.
The above isomorphisms allows us to identify the $L^+G$-torsors corresponding to the two displays
$\D$ and $\D^\Diamond\times_A B$ over $B$ as two $L^+G$-equivariant subschemes of the $L^+G'$-torsor 
corresponding $\D'\times_A B$ 
which are the same after base changing by $B\to \hat B$.
Using that $G'/G$ is represented by an affine scheme  (\cite[6.12]{ColliotSansuc}) we can see that this implies that these subschemes
 are equal.
 This gives an isomorphism 
$$
(i(\D^\Diamond),\D^\Diamond\times_AB,i(\id_{\D^\Diamond})\times_AB)
\cong(\D',\D,\psi),
$$
which shows that $(\D',\D,\psi)$ is effective.
\end{proof}

\section{Rapoport-Zink spaces}\label{sec:RZ}

We now give our definition of the Rapoport-Zink stack and functor, state
the representability conjecture and prove our main result on representability
in the Hodge type case.

\subsection{Local Shimura data}

Let $G$ be   a connected 
reductive group scheme  over $ \Z_p $.
We follow \cite{RapoportVi} and \cite{HP2}.

\subsubsection{}
\label{general}

Let $([b],  \{\mu\} )$ be a pair consisting of:

\begin{itemize}
\item a $G(\bar K)$-conjugacy class $\{\mu\}$ of  cocharacters $\mu: \Gm_{\bar K}\to G_{\bar K}$,

\item  a $\sigma$-conjugacy class $[b]$ of elements  $b\in G(K)$; here, as usual, $b$ and $b'$
are $\sigma$-conjugate if there is $g\in G(K)$ with $b'=gb\sigma(g)^{-1}$.
\end{itemize}

  We let $E \subset \bar{K}$ be the field of definition of the conjugacy class $\{\mu\}$. This is the 
\emph{local reflex field.}
Denote by $\co_E$ its valuation ring and by $k_E$ its (finite) residue field. 
In fact, under our assumption on $G$, the field $E\subset \bar K$ is 
contained in $K$ and there is a cocharacter 
$
\mu: \Gm_E\to G_E
$ 
in the conjugacy class $\{\mu\}$ that is  defined over $E$; see \cite[Lemma (1.1.3)]{KottTwisted}.
In fact, we can find a representative  $\mu$ that extends to an integral cocharacter 
\begin{equation}\label{integral mu rep}
\mu: \Gm_{\co_E}\to G_{\co_E},
\end{equation}
and the $G(\co_E)$-conjugacy class of such an $\mu$ is well-defined. 
In what follows, we usually assume that $\mu$ is such a representative.
 We can identify $\co_E$ with the ring of Witt vectors $W(k_E)$ and we have $E=W(k_E)[1/p]$.

We  write $\mu^\sigma=\sigma(\mu)$ for the Frobenius conjugate of (\ref{integral mu rep}).

\begin{definition} (cf.~\cite[Def. 5.1]{RapoportVi})
A  {\it local  unramified  Shimura datum}  is a triple 
$(G, [b], \{\mu\})$, in which  $G$ is a connected reductive group
over $\Z_p$, the pair $([b], \{\mu\})$ is as above, and we assume
\begin{enumerate}
\item  $\{\mu\}$ is minuscule,  

\item  for some (equivalently, any)   integral representative  (\ref{integral mu rep})  of $\{\mu\}$,
the $\sigma$-conjugacy class $[b]$ has a representative  
\begin{equation}\label{b coset}
 b\in G(W)\mu^\sigma(p)G(W).
 \end{equation}
\end{enumerate}  
 
\end{definition}

By \cite[Theorem 4.2]{RapoportRich}, assumptions (i) and (ii) imply that $[b]$ lies in the set $B(G_{\q_p}, \{\mu\})$ 
of neutral acceptable elements for $\{\mu\}$; see \cite[Definition 2.3]{RapoportVi}. 
In particular, $(G_{\q_p}, [b], \{\mu\})$ is a local Shimura datum in the sense of  \cite[Def. 5.1]{RapoportVi}.

\subsection{Definitions and a representability conjecture}\label{defRZ}

Fix a local  unramified  Shimura datum $(G, [b], \{\mu\})$ and choose
an integral representative $\mu$ of the conjugacy class $\{\mu\}$ as 
in (\ref{integral mu rep}). Choose a representative $b$ of the $\sigma$-conjugacy class
$[b]$ such that $b=u\mu^\sigma(p)$, with $u\in G(W)$.
Consider the $(G,\mu)$-display $\D_0=(P_0, Q_0, u_0)$
over $k$ given by $P_0=L^+G_W$, $Q_0=H^\mu$, and $u_0: H^\mu\to L^+G_W$ given as the composition of the inclusion followed by
left multiplication by $u$.

Suppose that $S$ is a scheme in ${\rm Nilp}_W$ and consider pairs $(\D, \rho)$ with
\begin{itemize}
\item $\D$ a $(G,\mu)$-display over $S$,
\item $\rho: \bar\D \dashrightarrow \D_0\times_k \bar S $ a $G$-quasi-isogeny over $\bar S$.
\end{itemize}
(In the above, $\bar S$ and $\bar\D$ denote the reductions of $S$ and $\D$ modulo $p$.)
Consider the natural notion of isomorphism between two such pairs.

We will denote by ${\calRZ}_{G, \mu, \D_0}$  the  stack of groupoids over ${\rm Nilp}_{W}$ that classifies pairs $(\D, \rho)$ as above. We can see that this is a fpqc stack.
Denote by $\RZ_{G,\mu, b}$ the corresponding functor ${\rm Nilp}_W\to {\rm Sets}$ which sends
$S$ to the isomorphism classes of pairs $(\D,\rho)$ over $S$ as above.

Consider the group of automorphisms $ {\rm Aut}(\D_0[1/p])$ of the $G$-isodisplay given by $\D_0$. We can see that
\[
{\rm Aut}(\D_0[1/p])\simeq J_b(\q_p):=\{j\in G(L)\ |\ j^{-1}b\sigma(j)=b\}.
\]
This group acts ${\calRZ}_{G, \mu, \D_0}$ and on the functor $\RZ_{G,\mu, b}$ on the left by
\[
j\cdot (\D, \rho)=(\D, j\cdot \rho).
\]

We can now state:
 
\begin{conjecture}\label{conjmain}
 Assume that $-1$ is not a  slope of ${\rm Ad}^G(b)$.
The functor $\RZ_{G,\mu, b}$ 
is representable by a formal scheme which is formally smooth and formally locally of finite type
over $W$.
\end{conjecture}

Here,  representability by a formal scheme is in the sense explained in \cite{RapZinkBook}.
For $G=\GL_n$ and for $b$ with no slopes equal to $0$, the conjecture follows by a combination of the results of Rapoport-Zink and Zink and Lau: By Lau and Zink and the discussion of \ref{ZinkDisplays} (see also  \ref{glnexample}, \ref{glnQiso} and \ref{stressful01}) the functor for $\GL_n$ is equivalent to the Rapoport-Zink functor of deformations up to quasi-isogeny of the
$p$-divisible group over $k$ that corresponds to $\D_0$.  
For local unramified Shimura data of Hodge type, we will prove this conjecture for the restriction of the functor $\RZ_{G,\mu, b}$ to locally Noetherian schemes in ${\rm Nilp}_W$, see Theorem \ref{hodgemain}.

For general $(G, \mu)$ and $b$ as above, it follows easily from Proposition \ref{GNN01} that, if $B$ is Noetherian, then the objects of the groupoid ${\calRZ}_{G,\mu,\D_0}(B)$ have no automorphisms.

\begin{remark}\label{remarkDescent}
We refer to \cite[Def. 3.45]{RapZinkBook} for the notion of a Weil descent datum relative to the extension $W/\O_E$ for a functor on the category ${\rm Nilp}_W$. Set $f=[E: \q_p]$.
(Recall that $E/\q_p$ is finite and unramified.) Denote by $ \tau: \Spec(\O_E)\to \Spec(\O_E)$ the morphism induced by     $\sigma^f: \O_E\to \O_E$ and by $\bar\tau$ its reduction $\bar\tau: 
\Spec(k_E)\to \Spec(k_E)$. Using the construction of loc. cit. 3.48, we can define a Weil descent datum on $\RZ_{G,\mu, b}$ (relative to $\tau$). This datum is given by the isomorphism of functors $\alpha: \RZ_{G,\mu, b}\to  \RZ_{G,\mu, b}^{ \tau }$ (see loc. cit.) obtained by sending the pair $(\D, \rho)$ to the pair of $ \tau^*\D$ together with the $G$-quasi-isogeny 
\[
\bar\tau^*\bar\D\dashrightarrow \bar\tau^*(\D_0\times_k \bar S)\dashrightarrow \D_0\times_k \bar S
\]
where the first arrow is $\bar\tau^*(\rho)$ and the second arrow is given by the relative Frobenius of $\D_0$. 
\end{remark}

\subsubsection{} Suppose that $(\D, \rho)$ is as above with $S=\Spec(R)$ affine and $\D$ banal;
then $\D$ is determined by $U\in L^+G(R)$; the corresponding  $G$-isocrystal
over $\bar R$ is given by $U\mu^\sigma(p)$. The $G$-quasi-isogeny $\rho$ is given by left multiplication by $g\in LG(\bar R)$ which satisfies the identity
\begin{equation}
bF(g)=u\mu^\sigma(p)F(g)=gU\mu^\sigma(p)
\end{equation}
in $LG(\bar R)=G(W(\bar R)[1/p])$. Note that since $p$ is nilpotent in $R$, the ideal $W((p))\subset W(R)$ is $p$-power torsion  and so
\[
W(R)[1/p]=W(\bar R)[1/p].
\]
We conclude that in the banal case, the pair $(\D, \rho)$ is given by 
a pair $(U, g)\in L^+G(R)\times LG(R)$ which satisfies 
\begin{equation}\label{equationUg}
g^{-1}bF(g)=U\mu^\sigma(p).
\end{equation}
By the definitions, two pairs $(U', g')$, $(U, g)$ give isomorphic pairs $(\D, \rho)$, $(\D',\rho')$, 
when there exists $h\in H^\mu(R)$ such that
 \begin{equation}
(U', g')=(h^{-1}\cdot U\cdot  \Phi_{G,\mu}(h), g\cdot h).
  \end{equation}
  This implies that $\RZ_{G,\mu, b}$ is given by the isomorphism classes of objects of the
  (fpqc, or \'etale) quotient stack 
  \[
 [ (L^+G\times_{LG,\mu, b} LG)/ H^\mu].
  \]
 Here the fiber product is 
  $$
\xymatrix{{L^+G\times_{LG,  \mu, b}LG}\ar[r]\ar[d] & {LG}\ar[d]^{c_{b}}\\
{L^+G}\ar[r]&{LG}}
$$
with $c_{b}(g):=g^{-1}bF(g)$ and the bottom horizontal map is the natural
map followed by right multiplication by $\mu^\sigma(p)\in G(E)\subset LG(R)$. 
The quotient is for the action of $H^\mu$ given by
\[
(U, g)\cdot h=(h^{-1}\cdot U\cdot  \Phi_{G,\mu}(h), g\cdot h).
\]

\subsubsection{}\label{deformRZ}
 Continue with the set-up above and assume that $I\subset R$ is a nilpotent ideal.
Set $R_0=R/I$. Then, since $W(R)[1/p]=W(R_0)[1/p]$, the pair $(U, g)$ is determined by $(U, g_0)$.
We can use this to deduce that for any $(\D_0,\rho_0)$ over $R_0$, the forgetful functor 
$\calRZ_{G,\mu,\D_0}\to \B(G,\mu)$, $(\D,\rho)\mapsto \D$,
induces an equivalence of deformation functors ${\rm D}_{(\D_0,\rho_0), S}\xrightarrow{\sim} {\rm D}_{\D_0, S}$.
In particular, our results in \S \ref{deform04} apply to the deformation theory of $\calRZ_{G,\mu, \D_0}$.

\subsubsection{} Let $k'$ be an algebraically closed field extension of $k$ and set
 $W'=W(k')$, $K'=W'[1/p]$. In this case, since $W(k')$ is torsion free, the equation $g^{-1}bF(g)=U\mu^\sigma(p)$
shows that $U$ is determined from $g$.
By the above discussion, we have
\begin{equation}\label{RZk}
\RZ_{G,\mu, b}(k')=\{g\in G(K')\ |\ g^{-1}bF(g)\in G(W')\mu^\sigma(p)\}/H^\mu(k'),
\end{equation}
where $H^\mu(k')\subset G(K')$ acts on $G(K')$ on the right.
Since $k'$ is perfect, we have $I(k')=pW(k')$. This gives
\[
H^\mu(k')=G(W')\cap \mu(p)^{-1}G(W')\mu(p).
\]
(Hence, the group $H^\mu(k')$ is equal to group of  $W'$-valued points of a parahoric subgroup scheme
defined over $W(k_0)$.)

\begin{proposition}\label{ADL} We have a bijection
\[
\RZ_{G,b,\mu}(k')\cong \{g\in G(K')\, |\, g^{-1}bF(g)\in G(W')\mu^\sigma(p)G(W')\}/G(W')
\]
where the quotient is  for the natural right action of $G(W')$ on $G(K')$.
\end{proposition}

The right hand side of the identity in this statement is, by definition, the affine Deligne-Lusztig set 
\[
X_{G,\mu^\sigma, b}(k') \subset G(K')/G(W')
\]
 for the data $(G,\mu^\sigma, b)$ and the field $k'$.

\begin{proof}
This follows from \ref{RZk} and an argument as in the proof of Proposition \ref{displaysoverk} (or \cite[Proposition 2.4.3 (i)]{HP2}) which shows that 
the map 
\[
G(K')/H^\mu(k')\to G(K')/G(W'); \ \ gH^\mu(k')\mapsto gG(W'),
\]
restricts to a bijection between the set in the right side of \ref{RZk} and $X_{G,\mu^\sigma, b}(k')$.
\end{proof}

\begin{remark}
By \cite[Proposition 2.4.3 (iii)]{HP2} , the map $g\mapsto \sigma^{-1}(b^{-1}g)$ gives a bijection 
$X_{G,\mu^\sigma, b}(k')\xrightarrow{\sim} X_{G, \mu, b}(k')$ and so we also have
\[
\RZ_{G, \mu, b}(k')\cong X_{G, \mu, b}(k').
\]
\end{remark}

\section{Rapoport-Zink spaces of Hodge type}

\subsection{Hodge type local Shimura data}

The two definitions below are slight variants of definitions in \cite{HP2}.

\begin{definition}

The local unramified Shimura datum  $(G, [b], \{\mu\})$  is of {\sl Hodge type} if there exists a closed group scheme embedding 
$
\iota : G\hookrightarrow \GL( \Lambda),
$
for  a free $\Z_p$-module $ \Lambda$ of finite rank, with the following property: 
After a choice of basis $ \Lambda_{\co_E}\iso  \co_E^n$, the composite
\[
\iota \circ \mu: \Gm_{\co_E}\to {\rm GL}_{n, \co_E}
\] 
is the minuscule cocharacter $\mu_{r,n}(a)={\rm diag}(1^{(r)}, a^{(n-r)})$, for some $1\leq r< n$. 
 \end{definition}

 \begin{definition}\label{def:shimura-hodge} 
Let $(G, [b], \{\mu\})$ be a  local unramified Shimura datum  of Hodge type. 
A \emph{local  Hodge embedding datum} for $(G, [b], \{\mu\})$ consists of 
\begin{itemize}
\item
 a group scheme embedding $\iota : G\hookrightarrow \GL( \Lambda)$ as above, 
\item
the $G(W)$-$\sigma$-conjugacy class $\{gb\sigma(g)^{-1} : g\in G(W)\}$ of a representative $b \in G(W)\mu^\sigma(p)G(W) $  of $[b]$  as in (\ref{b coset}).  
\end{itemize}
\end{definition}

Notice that the  $G(W)$-conjugacy class of an integral representative $\mu$ as in  (\ref{integral mu rep}) is  determined from $\{\mu\}$.

We will refer to the quadruple $(G, b, \mu,  \Lambda)$, where $\mu$ is given up to $G(W)$-conjugation, and $b$ up to $G(W)$-$\sigma$-conjugation, as a {\sl local unramified Shimura-Hodge datum}.

 The following is the main result of the paper:
 
 \begin{theorem}\label{hodgemain}
 Assume that  $(G, [b], \{\mu\})$  is a local unramified Shimura datum of Hodge type
with a local Hodge embedding datum such that $\iota(b)$ has no slope $0$. Then the restriction of the functor $\RZ_{G,\mu, b}$ to 
 locally Noetherian schemes in ${\rm Nilp}_W$
is representable by a formal scheme which is formally smooth and formally locally of finite type
over $W$.
 \end{theorem}

\subsubsection{}
Fix a  local unramified  Shimura-Hodge datum  $(G, b, \mu,  \Lambda)$. The following
is obtained similarly to \cite[Lemma 2.2.5]{HP2}. Note however that, here, we are using the covariant
Dieudonn\'e module.

\begin{lemma}\label{lemma121}(\cite{HP2})
There is a unique,  up to isomorphism,   $p$-divisible group 
\[
X_0=X_0(G, b, \mu,  \Lambda)
\] 
over $k$ whose  (covariant) Dieudonn\'e module   is  $\Db(X_0)(W)=\Lambda\otimes_{\Z_p}W$ 
with Frobenius $F=b \circ \sigma$.  Moreover,  the Hodge filtration 
\[
V D_k\subset  D_k=\Db(X_0)(k)
\] 
is   induced by   a conjugate of the reduction $\mu_k: \Gm_k\to G_k $ of $\mu$ modulo $(p)$.
\end{lemma} 

 In what follows, we will show  Theorem \ref{hodgemain}. We will use a natural morphism from ${\rm RZ}_{G,\mu, b}$ to the functor represented by the ``classical''
 Rapoport-Zink formal scheme $\RZ_{X_0}$, where $X_0=X_0(G, b, \mu,  \Lambda)$ is as above.

\subsection{The proof of the representability theorem} 

\subsubsection{}
In this subsection we show Theorem \ref{hodgemain}.
We are going to use the notion of Frobenius separatedness defined in Appendix C. We write ${\rm ANilp}_\O^{\rm noeth}$, ${\rm ANilp}_\O^{\rm ared}$, ${\rm ANilp}_{\O}^{\rm aFs}$ be the full 
subcategories of ${\rm ANilp}_\O$ consisting of $\O$-algebras which are resp. Noetherian, resp. with nilpotent nilradical, resp. which are almost Frobenius separated.

\subsubsection{} In what follows we fix a closed group scheme immersion  
\[
i:(G,\mu)\hookrightarrow(G',\mu')
\]
as in \ref{embeddingpair}. We will eventually apply the following statements to the case that $G'=\GL_h$, $\mu'=\mu_{d, h}$, and $i$ is given by a Hodge embedding datum. 

\begin{proposition} 
\label{GNN02}
Suppose 
that $B$ is in ${\rm ANilp}^{\rm aFs}_{W}$.  Let $\D_1$ and $\D_2$ 
be banal displays with $(G,\mu)$-structure over $B$ such that $i(\D_1)$ 
and $i(\D_2)$ are adjoint nilpotent. Then a pair $(\phi,\psi)$ with
\begin{itemize}
\item[(i)]
$\phi:\D_1\dashrightarrow\D_2$ a $G$-quasi-isogeny, 
\item[(ii)]
$\psi:i(\D_1)\rightarrow i(\D_2)$ a $G'$-isomorphism,
\end{itemize}
is induced from a $G$-isomorphism $\epsilon:\D_1\rightarrow\D_2$ 
if and only if 
\[
i(\phi)=\psi[1/p],
\]
i.e. if and only if $i(\phi)$ and $\psi$ give  the same 
$G'$-quasi-isogeny $i(\D_1)\dashrightarrow i(\D_2)$.
(By Proposition \ref{GNN01}, the isomorphism $\epsilon$ is then uniquely determined.)
\end{proposition}
\begin{proof}
Using Lemma \ref{deform02} we can reduce to the case $pB=0$.  Again  set $J_n=\ker(F^n_{B})=\{x\in B |x^{p^n}=0\}$. 
Since the displays $\D_i$ are banal, they are represented by elements $U_i\in G(W(B))$. The quasi-isogeny 
$\phi$ gives us an element $k\in G(W(B)_\q)$ such that 
$U_2=k^{-1}U_1\Phi^{\mu}(k)$ holds in $G(W(B)_\q)$. The isomorphism 
$\psi$ gives us an element $l\in H^{\mu'}(B)$ such that 
$i(U_2)=l^{-1}i(U_1)\Phi^{\mu'}(l)$ holds in $G'(W(B))$. By our assumption,  the images of 
$i(k)$ and $l$ agree in the group $G'(W(B)_\q)$. 

Since $i$ is a 
closed immersion there exists a smallest ideal $I_0$ of $W(B)$ such that 
the restriction of the $W(B)$-valued point $l$ to $\Spec (W(B)/I_0)$ 
factors through, say $h\in G(W(B)/I_0)$. Since $i: G\to G'$ is of finite presentation, $I_0$ is a finitely 
generated ideal of $W(B)$.

The existence of $k$ shows that $I_0$ is contained in $W(B)[p^\infty]=\bigcup_nW(J_n)$. By the finite generation of $I_0$,
we can choose  a large enough $n$ such that 
$I_0\subset W(J_{n})$ and we can consider $h_{n}\in G(W(B)/W(J_{n}))=G(W(B/J_{n}))=L^+G(B/J_{n})$. Since $H^{\mu}=L^+G \cap H^{\mu'}$,  we see that $h_{n}$ 
is in  $H^{\mu}(B/J_{n})$. Let $U_{i, n}$ be the images of $U_i$  in the group $G(W(B/J_{n}))$. Note that the elements $U_{2, n}$ 
and $h_{n}^{-1}U_{1, n}\Phi^{\mu}(h_n)$ are well-defined in $G(W(B/J_{n}))$ and have the 
same image in $G'(W(B)_\q)$. Since $G'$ is of finite type, there exists another integer $n'\geq n$ such that these elements agree already in 
$G'(W(B/J_{n'}))$ and so also in $G(W(B/J_{n'}))$. Therefore, we have constructed an isomorphism
$$
\eta:\D_1\times_BB/J_{n'}\xrightarrow{\sim}\D_2\times_BB/J_{n'}
$$
with $\eta[1/p]=\phi$ and $i(\eta)\equiv\psi\mod J_{n'}$.  Now 
use deformation theory (as in the proof of Lemma \ref{deform02}) to extend $\eta$ to a compatible system of isomorphisms $\eta_N$ over $B/J_{n'}^N$, for all $N$, which are given by $h_{N}\in H^\mu(W(B/J_{n'}^N))$ and which, when viewed in $H^{\mu'}(W(B/J_{n'}^N))$, are all the reduction of a single element $h'\in H^{\mu'}(W(B))$. 
Set $I=\cap_{N} J_{n'}^N$; by our assumption that $B$ is almost Frobenius separated, $I$ is nilpotent. It follows that $h'\,{\rm mod}\, I$ is in $H^\mu(W(B/I))$ and so it gives a lift of $\eta$ to an isomorphism 
$\tilde\eta:\D_1\times_BB/I\xrightarrow{\sim}\D_2\times_BB/I
$ with $\tilde\eta[1/p]=\phi$ and $i(\tilde\eta)\equiv\psi\mod I$. 
We can now conclude by a similar deformation theory argument as before.
\end{proof}

\subsubsection{}\label{RZembedding} Now let $\D_0$ be a $(G,\mu)$-display defined over an algebraically closed field extension
$k$ of $k_0$. Suppose that the $(G',\mu')$-display $i(\D_0)$ is adjoint nilpotent, then the same holds true for
$\D_0$. We can consider the Rapoport-Zink stacks of groupoids ${\calRZ}_{G, \mu, \D_0}$, resp. ${\calRZ}_{G',\mu',i(\D_0)}$, of pairs of $(G,\mu)$-displays, resp. $(G',\mu')$-displays, together with a $G$-quasi-isogeny to $\D_0$, resp. a $G'$-quasi-isogeny to $i(\D_0)$. There is a natural transformation 
\[
i: {\calRZ}_{G,\mu, \D_0}\to {\calRZ}_{G',\mu', i(\D_0)}.
\]
(Note that by Lemma \ref{GNN01}, for $B$ in ${\rm ANilp}^{\rm aFs}_{W(k)}$, the objects
of ${\calRZ}_{G,\mu, \D_0}(B)$, ${\calRZ}_{G',\mu', i(\D_0)}(B)$, have no automorphisms.)

\begin{proposition}
\label{analog}
 Suppose that $B$ 
is an object of  ${\rm ANilp}_{W(k)}^{\rm ared}$ and let $A\subset B$ a Noetherian 
$W(k)$-subalgebra. Suppose we are given:
\begin{itemize}
\item $(\D,\delta)$ is an object of ${\calRZ}_{G,\mu,\D_0}(B)$, 
\item $(\D',\delta')$ is an object of ${\calRZ}_{G',\mu', i(\D_0)}(A)$,
\item $\psi$ an isomorphism $\psi: i((\D,\delta))\xrightarrow{} (\D',\delta')\times_A B$.
\end{itemize}
Then there is an object $(\D^\Diamond, \delta^\Diamond)$ of 
${\calRZ}_{G, \mu, \D_0}(A)$ together with isomorphisms
\[
(\D^\Diamond, \delta^\Diamond)\times_A B\xrightarrow{\sim} (\D, \delta),\quad 
i((\D^\Diamond, \delta^\Diamond))\xrightarrow{\sim} (\D',\delta'),
\]
which are compatible with $\psi$ in the appropriate manner.
\end{proposition}

\begin{proof} Given the above, the data $(\D',\D,\psi)$ give an $(A,B)$-display block as in Definition \ref{block}.
Part of the desired conclusion is that $(\D',\D,\psi)$ is effective. In fact, we will first show that 
it is enough to show this effectivity: Indeed, assume we have that, i.e. suppose that we have constructed
a $(G,\mu)$-display $\D^\Diamond$ over $A$ together with isomorphisms $i(\D^\Diamond)\xrightarrow{\sim} \D'$,
$\D^\Diamond\times_A B\xrightarrow{\sim}  \D$, which are compatible with $\psi$. We can then give the $G$-quasi-isogeny $\delta^\Diamond$ over $\bar A$ as follows:  The data $\delta$, $\delta'$ together with the above isomorphisms give a $G'$-quasi-isogeny $\tilde\delta': i(\bar \D^\Diamond)\dashrightarrow i(\D_0)\times_k \bar A$ and a $G$-quasi-isogeny $\tilde\delta: \bar \D^\Diamond\times_{\bar A} {\bar B}\dashrightarrow \D_0\times_k {\bar B}$. These two are compatible in the sense that
\[
i(\tilde\delta)=\tilde\delta'\times_{\bar A}{\bar B}.
\]
We claim that this implies that $\tilde\delta$ is defined over $\bar A$ 
and so it gives the desired $\delta^\Diamond$. To see this we can assume
that the display $\D^\Diamond$ is banal (by using \'etale
descent on $A$ and Corollary \ref{torsCor}). Then $\tilde\delta$, $\tilde\delta'$ are given 
by elements in $g\in G(W(B)[1/p])$ and $g'\in G'(W(A)[1/p])$ respectively and the condition is that
these elements coincide in $G'(W(B)[1/p])$. Now observe that our conditions on $A\subset B$ imply
\[
W(A)[1/p]\subset W(B)[1/p].
\]
(Indeed, since $A$, $B$ are in ${\rm ANilp}_W^{\rm ared}$ there is an integer $n$ such that $p^n$ annihilates 
both $W(\rad(A))$ and $W(\rad(B))$. Hence, $W(A)[1/p]=W(A_{\rm red})[1/p]$, $W(B)[1/p]=W(B_{\rm red})[1/p]$.
Since $W(A_{\rm red})$, $W(B_{\rm red})$ are $p$-torsion free and $A\subset B$ induces $A_{\rm red}\subset B_{\rm red}$, the inclusion follows.) Now $g$ is given by $\O_{G}\to W(B)[1/p]$ and $g'$ by $\O_{G'}\to W(A)[1/p]$. The compatibility condition is that the following diagram commutes
\[
\begin{matrix} 
\O_{G'} & \xrightarrow{ } & \O_G\\
g'\downarrow\ \ && \ \ \downarrow g\\
W(A)[1/p]&\hookrightarrow & W(B)[1/p].
\end{matrix}
\]
 Since $\O_{G'}\to \O_G$ is surjective this implies that $g$ factors through $W(A)[1/p]$;
 this provides the $G$-quasi-isogeny $\delta^\Diamond$ .

To show $(\D',\D, \psi)$ is effective we now proceed in several steps, some of which are similar to 
\cite[proof of Proposition 5.4.(ii)]{E7}. As it turns out, the two most important cases
that we need to handle are when $B$ is faithfully flat over $A$ (Step 1),
and when $A=K[[t]]$ and $B=K((t))$, with $K$ a field (Step 3). 

Observe that the statement 
holds for a finite product $A=\prod_{i=1}^nA_i$ if and only if 
it holds for each factor $A_i$, so that it does no harm to assume 
the connectedness of $\Spec A$ in each of the steps below.\smallskip
 
{\sl Step 1.} $B$ is faithfully flat over $A$ (e.g. if $A$ is a field).
\smallskip

 By Lemma \ref{torsCor}, there is an \'etale faithfully flat $B\to B'$ such that $\D\times_BB'$ is banal.
Notice that $B'$ has also nilpotent radical and so we can replace $B$ by $B'$ and assume that $\D$
is banal. Then the conclusion will be a  consequence of fpqc descent for $A\to B$ and ${\calRZ}_{G,\mu, \D_0}$.
We can use Corollary \ref{GNN05} and Proposition \ref{GNN02} applied to $B\otimes_A B$ to construct the descent datum as follows:

Let $d_1: B\to B\otimes_AB$, $d_2: B\to B\otimes_AB$, be the coprojections and $d: A\to B\otimes_A B$ their 
common restriction to $A$. Consider 
\[
d_i^*(\D,\delta):=(\D,\delta)\times_{B,d_i}(B\times_A B),
\]
for $i=1,2$, in $\calRZ_{G,\mu, \D_0}(B\otimes_AB)$. Define a $G$-quasi-isogeny $\phi:  d_1^*\D\dashrightarrow d_2^*\D$
by setting
\[
\phi=d_2^*(\delta)^{-1}\cdot d_1^*(\delta): d_1^*\D  \dashrightarrow d^*\D_0\dashrightarrow d_2^*\D.
\]
There is also a $G'$-isomorphism $\psi':  d_1^*i(\D)\to d_2^*i(\D)$ given by the composition
\[
d_1^*i(\D)  \xrightarrow{d_1^*(\psi)} d_1^*(\D'\times_AB)=d^*\D' =d_2^*(\D'\times_A B) \xrightarrow{d_2^*(\psi)^{-1}} d_2^*i(\D).
\]
We can now see that $\psi'$ gives the $G'$-quasi-isogeny $i(\phi)$. By Corollary \ref{GNN05},
$B\otimes_A B$ is in ${\rm ANilp}_W^{\rm aFs}$ and so we can apply Proposition \ref{GNN02}
to construct the descent datum which is given by a $G$-isomorphism $\epsilon: d_1^*\D  \xrightarrow{\sim}  d_2^*\D$
over $B\otimes_A B$. The $G$-isomorphism $\epsilon$ is compatible with the  $G$-quasi-isogenies $d_1^*\delta$ and $d_2^*\delta$.\smallskip

{\sl Step 2.}  $A$ is a complete local Noetherian ring and $B$ is finite over $A$.
\smallskip

First let us make the following observation: 
Let $A\to A'$ be an \'etale extension of local rings with a finite separable extension of residue fields $k'$
over the residue field $k=A/\frakm$ (then $A'$ is also complete Noetherian). Assume we have the result for $A'\hookrightarrow B':=B\otimes_AA'$, and
for the base changes of our displays to $A'$
($B'$ is then also in ${\rm ANilp}_W$). Then we can apply descent (or Step 1) to the faithfully flat $A\to A'$ to 
further descent to $A$ and obtain the result for $A\to B$. 

Now apply Step 1 to the inclusion $A/\frakm\hookrightarrow B/J(B)$. This shows that the display block
$(\D'\times_AA/\frakm,\D\times_BB/J(B),\psi\times_BB/J(B))$ is effective and is given by a $(G,\mu)$-display over $k=A/\frakm$. There is a finite separable extension $k'$ of $k$ such that this display is banal
and we can find an \'etale extension $A'$ as above with $k'$ as residue field.
Now apply part (ii) of Lemma \ref{deform01} to the base changes by $A'$ to descent to $A'$ and conclude by using the observation above.\smallskip

{\sl Step 3.} $A=K[[t]]$ and $B=K((t))$, where $K$ is a field extension of $k$.
\smallskip

By using Step 1 we may assume that $K$ is algebraically closed. By base changing to 
a finite separable extension $L$ of $K((t))$ we can arrange that the $(G,\mu)$-display $\D$
is banal. Notice that if $A$ is the integral closure of $K[[t]]$ in $L$, then $A=K[[u]]$
for some variable $u$ and $L=K((u))$. Since $A=K[[u]]$ is finite over $K[[t]]$, an application of 
Step 2 shows that we can always descend along $K[[t]]\hookrightarrow K[[u]]$; hence, 
we can reduce to the situation $A=K[[t]]\hookrightarrow B=K((t))$ with $K$ algebraically closed
and all displays banal. For simplicity of notation, set $R=K[[t]]$, $E=K((t))$.
Hence we can find   $U'\in G'(W(R))$ 
for $\D'$, and $U\in G(W(E))$ for $\D$.
The isomorphism  $\psi: i(\D)\xrightarrow{}\D'\times_RE$ 
is given by $l\in H^{\mu'}(E)\subset G'(W(E))$ with $U=l^{-1}U'\Phi_{G,\mu}(l)$. Using the two isogenies $\delta$, $\delta'$
we obtain a factorization 
\[
l=ab^{-1}
\]
in $G'(W(E)[1/p])$. In this, $b\in G(W(E)[1/p])$ gives  $\delta: \D\dashrightarrow \D_0\times_kK((t))$ and 
$a\in G'(W(R)[1/p])$ gives $\delta': \D'\dashrightarrow i(\D_0)\times_kR$. Consider the  quotient $G'/G$ which is represented by an affine scheme $Z$ (\cite[6.12]{ColliotSansuc}). 
The above identity implies that the $W(E)$-valued point  of $Z$ which is given by $l$ is equal to the $W(R)[1/p]$-valued point 
which is given by $a$ with the equality of points considered in $Z(W(E)[1/p])$.
Now use
\[
W(E)\cap W(R)[1/p]=W(R).
\]
(The intersection takes place in $W(E)[1/p]$.) Using that $Z$ is affine, we see that there is a $W(R)$-valued point $z$
of $Z$ which gives both $l$ and $a$. Since $G'\to Z=G'/G$ is a $G$-torsor and $W(R)=W(K[[t]])$
is henselian, there is $c\in G'(W(R))$ that lifts $z$. Then $c=a\cdot d$ with $d\in G(W(R))$. Thus we can adjust
$b$ by multiplying by $d^{-1}$ and assume now that
$
l=ab^{-1}
$
with $a\in G'(W(R))$ and $b\in G(W(E))$. Now use that, by the Iwasawa decomposition, we have
\[
G(E)=P_\mu(E)\cdot G(R).
\]
This and the surjectivity of $G(W(E))\to G(E)$, $G(W(R))\to G(R)$ 
(which holds by Hensel's lemma, since $W(E)$, resp. $W(R)$, is $I(E)$-adically, resp. $I(R)$-adically, complete
and $G$ is smooth)
gives
\[
G(W(E))=H^\mu(E)\cdot G(W(R)).
\]
Write $b=h\cdot g$ with $h\in H^\mu(E)$, $g\in G(W(R))$. Then 
$l=ab^{-1} =ag^{-1}\cdot h^{-1}$. Hence, we may write $l=ab^{-1}$
with $a\in G'(W(R))$ and $b\in H^\mu(E)$. Observe that then $a\in H^\mu(R)$.
Now set
\[
U^\Diamond=a^{-1}U'\Phi_{G', \mu'}(a)=b^{-1}U\Phi_{G, \mu}(b).
\]
This element is  in  both $G'(W(R))$ and $G(W(E))$, hence it belongs to 
the intersection $G(W(R))$. We can now see that it defines the desired display
$\D^\Diamond$; the elements $a$ and $b$ give the isomorphisms of 
$i(\D^\Diamond)$ to $\D$ and of $\D^\Diamond\times_RE$ to $\D$
respectively.\smallskip

In what follows, we denote by $Q(R)$ the total quotient ring of $R$, \emph{i.e.} the localization $N^{-1}R$ of $R$ at the set of non-zero divisors $N\subset R$.
\smallskip

{\sl Step. 4.} 
$A$ is a reduced complete local one-dimensional Noetherian ring.
\smallskip

Consider $A\subset Q(A)\subset B\otimes_A Q(A)$; we can apply Step 1
to the base change $Q(A)\subset B\otimes_A Q(A)$. This allows us to 
reduce to considering $A\subset Q(A)$, i.e. we can assume $B=Q(A)$.
Consider the normalization $A'$ of $A$ in $Q(A)$. By \cite[Th\'eor\`eme (23.1.5)]{egaiv} this is a finite 
extension of $A$. By the Cohen structure theorem we have $A'\simeq K[[t]]$, $Q(A)\simeq K((t))$.
The result follows by applying sucessively Step 3 and Step 2.
\smallskip

The sequence of the following four steps settles the case of Noetherian 
rings of finite Krull dimension, we argue by induction on $\dim(A)$:
\smallskip

{\sl Step 5.}  $A$ is an integrally closed complete local Noetherian ring.
\smallskip

By Serre's condition $S_2$, there exists a regular sequence of length two, i.e. elements $f$ and $g$ such that $f$ is  neither a unit nor a zero-divisor 
of $A$ and $g$ is neither a unit nor a zero-divisor of $A/fA$. Notice that 
$A[1/g]$ is a Noetherian ring of dimension strictly less than $\dim A$, 
so by induction we can apply Step 8 (for a ring of smaller dimension) to the inclusion 
$A[1/g]\subset B[1/g]$ and we obtain the desired display $\D^\Diamond$ over the ring $A[1/g]$. Thus 
we may replace $B$ by $A[1/g]$, and so do consider the inclusion $A\subset A[1/g]$, which 
satisfies the assumption of part (i) of Lemma \ref{deform01}. (Here, to make sure that
the $(G,\mu)$-display over the residue field $A/\frakm$ is banal, we might need to base change by a finite \'etale local extension
$A\to A'$ as in Step 2.)
\smallskip

{\sl Step 6.} 
$A$ is a complete local Noetherian ring.
\smallskip

Using Lemma \ref{deform05} we see that we can assume that $A$ is reduced.
Then by an argument as in Step 4, we can reduce to the case $B=Q(A)$.
Just as in 
the Step 4 we consider the normalization $A'$ of 
$A$ in $Q(A)$. Using the previous step we can replace $B$ by the ring $A'$. However, by  \cite[Th\'eor\`eme (23.1.5)]{egaiv}, $A'$ is a finite extension of $A$, so that we can conclude by applying Step 2.
\smallskip

{\sl Step 7.}
$A$ is a local Noetherian ring.
\smallskip

Just as in the step above we can assume $A$ is reduced, putting us into a position where $B$ may be replaced by $Q(A)$. 
Apply Step 6 to the ring extension $\hat A\subset Q(A)\otimes_A\hat A$, where $\hat A$ denote the completion of the 
local ring $A$. This allows us to reduce to the case
$A\to B=\hat A$. Since $A\to \hat A$ is faithfully flat we can conclude by applying Step 1.
\smallskip

{\sl Step 8.}
$A$ is a Noetherian ring of finite Krull dimension.
\smallskip

As before we can assume $A_{\rm red}=A$ and $B=Q(A)$. By base change, for every maximal ideal 
$\frakm$ of $A$ we obtain a $(A_\frakm,Q(A_\frakm))$-display block 
\[
(\D'\times_AA_\frakm,\D\times_{Q(A)}Q(A_\frakm),\psi\times_{Q(A)}Q(A_\frakm)),
\]
where $A_\frakm$ stands for the localization of $A$ at $\frakm$. 

Set $A^\sharp:=\prod_\frakm A_\frakm$, which is reduced ($\frakm$ runs through the set of maximal ideals). Observe that the fact that $A_\frakm$ are all reduced 
implies that the   union $\cup_{\frak m} \Spec(A_\frakm)\hookrightarrow \Spec(A^\sharp)$ of the 
closed immersions $\Spec(A_\frakm)\subset \Spec(A^\sharp)$ is dense in $\Spec( A^\sharp)$. 
(In fact, in general, if the radical of $\prod_{i} A_i$ is nilpotent, then $\cup_i\Spec(A_i)$ is dense in $ \Spec(\prod_i A_i)$.)

Now let us write  $(\D^\Diamond_\frakm, \delta^\Diamond_\frakm)$ 
for its descent to ${\calRZ}_{G,\mu, \D_0}(A_\frakm)$, of which the existence is granted by the previous step.
Now let us construct a product display $\D^\sharp:=(P^\sharp, Q^\sharp, u^\sharp)=\prod_\frakm\D_\frakm^\Diamond$ over the 
ring $A^\sharp$:

  i) We obtain the $L^+G$-torsor $P^\sharp$ by applying Remark \ref{torsor04}. To obtain the descent 
$Q^\sharp$ to an $H^\mu$-torsor use the observation \ref{torsordatum}: We need a section over $\Spec(A^\sharp)$ of the 
corresponding $X_\mu$-bundle $ P^\sharp/H^\mu$ for $X_\mu=G/P_\mu$. By assumption, we have such a section
over $\Spec(A_\frakm)$, for all $\frakm$, while the $X_{\mu'}$-bundle $P'/H^{\mu'}$ for $X_{\mu'}=G'/P_{\mu'}$ has a section over $\Spec(A)$ and therefore over $\Spec(A^\sharp)$. These  agree as sections of $P'/H^{\mu'}$ 
over $\Spec(A_\frak m)$, for all $\frakm$. Now notice that, by descent, $P^\sharp/H^\mu\subset P'^\sharp/H^{\mu'}$ is a closed immersion and hence, by the above density of $\cup_{\frak m} \Spec(A_\frakm)\hookrightarrow \Spec(A^\sharp)$,
we see that these give a section of $P^\sharp/H^\mu$ over $A^\sharp$. 

ii) The construction of $u^\sharp$
from $(u_\frakm)$ is obtained by an argument as in the proof of essential surjectivity in Lemma \ref{torsor03}.

  In addition, we need to construct a ``compatible'' $G$-quasi-isogeny $\delta^\sharp: \D^\sharp\dashrightarrow\D_0\times_k A^\sharp/(p)$:
Let us first assume that $\D'$, $\D^\Diamond_\frakm$ are banal and fix trivializations of the torsors
$P'$, $P^\Diamond_\frakm$. Then $(\delta^\Diamond_\frakm)$ is given by 
\[
(g_\frakm)\in \prod\nolimits_\frakm G(W(A_{\frakm})[1/p])=G(\prod\nolimits_\frakm (W(A_\frakm)[1/p]));
\]
this also 
lies in $G'(W(A^\sharp )[1/p])$ and therefore in $G(W(A^\sharp)[1/p])$. The non-banal case is treated in a similar way
by working with, instead of   points of affine group schemes, points of the affine schemes of 
suitable torsor isomorphisms.

Notice that the morphism $A\rightarrow A^\sharp$ is faithfully flat. 
By construction, $\D'\times_AA^\sharp\cong i(\D^\sharp)$, and so $(\D^\sharp, \delta^\sharp)$, $(\D',\delta')$,
provide data to which we can apply the special case of Step 1.
\smallskip

{\sl Step 9.} $A$ is an arbitrary Noetherian ring.
\smallskip

Since every local Noetherian ring has finite Krull dimension, we can apply the previous step one more time, because meanwhile 
we know the local result without a restriction on the dimension.
\end{proof}
 
\begin{corollary}\label{closedRZ}
Fix a Noetherian algebra $A$ in ${\rm ANilp}_W$. Let $(\D', \delta')$ be an object of the stack $\calRZ_{G', \mu', i(\D_0)}$
over $A$.  Then there exists an ideal $I\subset A$ such that the following statement is true:
For each $A$-algebra $f: A\to B$ in ${\rm ANilp}_W^{\rm ared}$, we have $f(I)=0$ if and only if 
there exists an object $(\D, \delta)$ of $\calRZ_{G,\mu, \D_0}$ over $B$ 
together with an isomorphism $\phi: (\D', \delta')\times_A B\xrightarrow{\sim} i((\D, \delta))$.
\end{corollary}

This implies that the morphism $\calRZ_{G,\mu, \D_0}\to \calRZ_{G',\mu', i(\D_0)}$ of stacks, when restricted over ${\rm ANilp}_W^{\rm noeth}$, is represented by a closed immersion.

\begin{proof}
This is similar to \cite[proof of Theorem 5.5]{E7}: Let us say $(\D', \delta')$ has 
$G$-structure over $f: A\to B$ if there is an isomorphism $\phi: (\D', \delta')\times_A B\xrightarrow{\sim} (\D, \delta)$.
Then Proposition \ref{analog} implies that $(\D',\delta')$ has a $G$-structure over the quotient $A/{\rm ker}(f)$.
Denote by $S{(\D',\delta')}$ the set of ideals of $A$ for which $(\D',\delta')$ has a $G$-structure 
over $A/I$. If $I$ and $J$ are in $S{(\D',\delta')}$ we can apply Proposition \ref{analog} to $A/(I\cap J)\hookrightarrow A/I\times A/J$ and deduce that $I\cap J$ is also  in $S{(\D',\delta')}$.
In general, finite intersections of ideals in $S{(\D',\delta')}$ are also in $S{(\D',\delta')}$.
Now consider the ideal
\[
{\mathfrak I}:={\mathfrak I}_{(\D',\delta')}=\cap_{I\in S{(\D',\delta')}} I.
\]
Consider the reduced product  $B^\flat=\prod_{I\in S{(\D',\delta')}} A/\rad(I)$ of the reduced rings $A/\rad(I)$.
We can construct a product $(G,\mu)$-display $\D$ and a $G$-quasi-isogeny of $\D$ to $\D_0$ over this product
ring $B^\flat$ by an argument as in the proof of Step 8 above. Now apply Proposition \ref{analog}
to 
\[
A/\rad(\mathfrak I)\hookrightarrow B^\flat=\prod_{I\in S{(\D',\delta')}} A/\rad(I).
\]
We obtain that $\rad(\mathfrak J)$ also belongs to $S(\D',\delta')$. 
Now repeat the argument and apply Proposition \ref{analog} to
\[
A/\mathfrak I\hookrightarrow \prod_{I\in S(\D',\delta'),\ \mathfrak I\subset I\subset \rad(\mathfrak I)} A/I.
\]
(Notice that the radical of this product is also a nilpotent ideal.) This implies that $\mathfrak I\in S(\D',\delta')$
which is enough to deduce the result.
\end{proof}
 
Apply the above to the case that $G'=\GL_h$, $\mu'=\mu_{d, h}$, and $i$ is given by a Hodge embedding datum. 
By the results of Rapoport-Zink, Zink, and Lau, the functor $\RZ_{\GL_h, \mu_{d, h}, i(b)}$ is representable by a $W$-formal scheme which is locally formally of finite type over $W$ (see \ref{defRZ}; notice that the base-point $p$-divisible group
is $X_0$, given in Lemma \ref{lemma121}).
By Corollary \ref{closedRZ}, we obtain that the restriction of $\RZ_{G,\mu, b}$
to ${\rm ANilp}^{\rm noeth}_W$ is represented by a $W$-formal  closed subscheme of 
$\RZ_{\GL_h, \mu_{d, h}, i(b)}$ which is then also formally locally  of finite type over $W$.
Formal smoothness over $W$ follows from our deformation theory results. This concludes the proof of Theorem 
\ref{hodgemain}.\qed
  
\begin{remark} We now easily see that, in the Hodge type case of Theorem \ref{hodgemain}, the $W$-formal scheme representing 
the restriction of
$\RZ_{G,\mu, b}$ above to locally Noetherian schemes is isomorphic to formal schemes
 constructed in \cite{KImRZ} and 
 \cite{HP2} (when these are defined, for example, when the local Hodge embedding is globally realizable, see loc. cit.)
Indeed, all of these are $W$-formal closed subschemes
of the classical Rapoport-Zink $W$-formal scheme $\RZ_{\GL_h, \mu_{d, h}, i(b)}$
 with the same $k$-valued points (given by the affine Deligne-Lusztig set, Proposition  \ref{ADL}), and the same formal completions
 at these points (by deformation theory,
see \ref{deformRZ} and \ref{deform04})
so they agree by flat descent. Let us note, however, a slight difference in notation:
The Rapoport-Zink formal scheme for $(G, \mu, b)$ in this paper agrees with the one 
for $(G, \delta\cdot\mu^{-1}, pb^{-1})$ in \cite{HP2}, with $\delta: \Gm\to G\subset \GL(\Lambda)$ the central diagonal torus. (The existence of $\delta$ is part of the assumption of being of Hodge type in loc. cit.). The reason for this discrepancy is that \cite{HP2} uses
the contravariant Dieudonn\'e functor for the construction 
of the base point
$p$-divisible group $X_0$ corresponding to $i(\D_0)$.
\end{remark}

\begin{remark} In the Hodge type case of Theorem \ref{hodgemain}, we can use the 
descent datum $\alpha$ given in Remark \ref{remarkDescent} to descend $\RZ_{G,\mu, b}$
over $\O_E$ as in \cite[3.49, 3.51]{RapZinkBook}. This is done by an argument as in loc. cit., see for example loc. cit. Lemma 3.50.
\end{remark}

\begin{appendix}

\setcounter{subsubsection}{0}

\section{Minuscule cocharacters and parabolics}\label{simplerversion}

Here we collect some notations and standard results on (minuscule) cocharacters and 
corresponding parabolic and unipotent subgroups of reductive group schemes.
We refer the reader to \cite{SGA3} or \cite{ConradNotesSGA} for more details.

\subsubsection{}\label{dynamicparabolic}
We consider a perfect field $k_0$ and a  smooth affine group scheme $G$ over 
$W(k_0)$ with connected fibers. Let $\mu:\g_{m,W(k_0)}\rightarrow G$ be a group scheme homomorphism
(a ``cocharacter'' of $G$). We consider the closed subgroup schemes $U_\mu$ and $P_\mu$ 
of $G$ which are defined by the following subfunctors of $G$ on $W(k_0)$-algebras
\[
P_\mu(R)=\{g\in G(R)\ | \ \lim_{t\mapsto 0}\mu(t)g\mu(t)^{-1}\, {\rm exists} \}
\]
\[
U_\mu(R)=\{g\in G(R)\ | \ \lim_{t\mapsto 0}\mu(t)g\mu(t)^{-1}=1\}
\]
(See \cite[2.1]{ConradGabberPrasad}, or \cite[Theorem 4.1.17]{ConradNotesSGA}.)

Here   ``$\lim_{t\to 0}$ exists'', by definition, implies that the conjugation action 
$\g_{W(k_0)}\times_{W(k_0)} P_\mu\to P_\mu$ extends to a morphism 
\[
\Int_\mu: \a_{W(k_0)}^1\times_{W(k_0)}P_\mu
\rightarrow P_\mu.
\] This gives an action 
 of the monoid scheme $\a_{W(k_0)}^1$ 
on  $P_\mu$ by group scheme endomorphisms.
Under this the zero section of $\a_{W(k_0)}^1$ maps $U_\mu$ to the neutral 
section.

\subsubsection{}\label{parabolicfacts}
 By \cite[Theorem 4.1.17]{ConradNotesSGA} we have:    

\begin{itemize}
\item
Under our assumption, $U_\mu$ and $P_\mu$ are smooth group schemes over $W(k_0)$
with connected fibers.  

\item The subgroup scheme $U_\mu$ is unipotent
and the multiplication
\[
m: P_{\mu}\times_{W(k_0)} U_{\mu^{-1}} \to G
\]
is an open immersion. We will denote by $G^*_\mu$ or $G^*$, if $\mu$ is clear from the context,
the open subscheme of $G$ given as the image of this morphism. 

\item  If $
\frakg=\Lie(G)=\oplus_{n\in \Z}\frakg_n$
is the weight space decomposition of the Lie algebra under the adjoint action
(i.e. $\frakg_n=\{v\in \frakg \ |\ \mu(t)v\mu(t)^{-1}=t^nv\}$), then we have
\[
\frakp_{\mu}=\Lie(P_{\mu})=\oplus_{n\geq 0} \frakg_n,
\]
\[
 \fraku_{\mu^{-1}}=\Lie(U_{\mu^{-1}})=\oplus_{n<0}\frakg_n.
\]
\end{itemize}

\subsubsection{}\label{splitcase}
Suppose that $G$ is connected split reductive over $W(k_0)$ and $T\subset G$ is a split maximal torus such that $\mu$ factors through $T$. Denote by $\Phi$ the roots of $G$ and for $a\in \Phi$ by $\frakg_a\subset \frakg$, resp. $U_a$, the root $W(k_0)$-subspace of $\frakg$, resp. root subgroup scheme  of $G$. Then $\frakg_n=\oplus_{a | \langle\mu, a\rangle=n}\frakg_a$. Denote by $\Phi(\mu)\subset \Phi$ the set of roots $a$ such that $\langle\mu, a\rangle>0$.
Then the multiplication (with the factors in the product taken in any order)
\[
\prod_{a\in \Phi(\mu)} U_a\to U_\mu
\]
gives an isomorphism of $W(k_0)$-schemes (\cite[5]{ConradNotesSGA}). 
The subgroup scheme $P_\mu\subset G$ is a parabolic subgroup. The group scheme $U_\mu$
 is the unipotent radical of $P_\mu$.  It contains
a finite filtration
\[
U_\mu=U_{\Phi(\mu)\geq 1}\supseteq U_{\Phi(\mu)\geq 2}\supseteq \cdots
\]
of normal subgroup schemes such that the multiplication map
\[
\prod_{a | \langle\mu, a\rangle= n} U_a\to U_{\Phi(\mu)\geq n}/U_{\Phi(\mu)\geq n+1}
\]
(with the factors in the product taken in any order) is an isomorphism of {\sl group} schemes
(\cite[Prop. 5.1.16]{ConradNotesSGA}).

Recall that $\mu$ defines a 
decreasing filtration ${\rm Fil}^\bullet(V)$ on each representation $G\to\GL(V)$.
We can view $P_\mu\subset G$ as the subgroup scheme 
that respects the   filtration ${\rm Fil}^\bullet(V)$. 
Then $U_\mu\subset P_\mu$ is the subgroup 
scheme of $P_\mu$ that acts trivially on the graded ${\rm gr}^\bullet(V)$.

\subsubsection{} \label{minusculecase}
We assume that $G$ is a connected reductive group
scheme over $W(k_0)$ with $k_0$ a finite field.  
Then there is a finite field extension $k_0'/k_0$ such that $G\times_{W(k_0)}W(k_0')$
is split.  We now assume that the cocharacter $\mu$ is minuscule, i.e. that $\langle \mu, a\rangle\in \{-1,0,1\}$ for all absolute roots $a\in \Phi$. 

\begin{lemma}\label{minusculeunipotent}
The unipotent group scheme $U_\mu$ is commmutative and is isomorphic to
${\mathbb G}_a^r\times_{\Z_p}W(k_0)$, where $r\geq 0$ and ${\mathbb G}_a=\Spec(\Z_p[T])$ is the additive group scheme over $\Z_p$.
\end{lemma}

\begin{proof}
First we see that $U_\mu\times_{W(k_0)}W(k_0')\simeq {\mathbb G}_a^r\times_{\Z_p}W(k_0')$: This follows from \cite[Prop. 5.1.16]{ConradNotesSGA} and \'etale descent (see also loc. cit. Theorem 5.4.3).
Indeed, since there are no absolute roots $a$ with $\langle \mu, a\rangle\geq 2$, the subgroup scheme $U_{\Phi(\mu)\geq 2}$ of loc. cit. is trivial. We can now conclude using \cite[XVII.4.1.5]{SGA3} and the fact that all projective finitely generated $W(k_0)$-modules
are free.
\end{proof}

\section{Loop group torsors}

Suppose that $R$ is a 
$W(k_0)$-algebra. As usual, $G$ is a connected reductive group scheme 
over $\Z_p$. We will compare between  $L^+G$-torsors over $R$
and $G$-torsors over $W(R)$. 
 Here, we view $W(R)$ as a $W(k_0)$-algebra
via $W(k_0)\to W(W(k_0))\to W(R)$.  The torsors are, by definition, locally trivial for the fpqc topology. Our convention is that the group acts on the right.

\begin{lemma}
\label{torsor01} If $P$ is a $G$-torsor 
over $W(R)$, then the Greenberg transform $FP$ is a $L^+G$-torsor over $R$. 
\end{lemma}
\begin{proof}
Observe that $P$ is affine of finite presentation and smooth by descent,  
and so by Proposition \ref{greenbergprop} the Greenberg transform $FP$ is affine flat and formally smooth over $R$. We can also easily see that $FP\to \Spec(R)$ is surjective, hence faithfully flat.
By Proposition \ref{greenbergprop} the action morphism $P\times_{W(R)}(G\times_{W(k_0)}W(R))\to P$ gives
an action
\[
FP\times_R L^+G \to FP.
\] Since $P$ is a $G$-torsor 
the morphism  $P\times_{W(R)}(G\times_{W(k_0)}W(R))
\xrightarrow{\sim} P\times_{W(R)} P$ given by $(x, g)\mapsto (x, x\cdot g)$
is an isomorphism. By Proposition  \ref{greenbergprop} again, the morphism  $FP\times_R L^+G\simeq FP\times_R FP$ given by the above action is an isomorphism and the result follows since $FP\to \Spec(R)$ is fpqc.
\end{proof}

\begin{proposition}
\label{torsor02}
Let $R$ be  in ${\rm ANilp}_{W(k_0)}$.

a) If $P$ is a $G$-torsor over $W(R)$, then there is an \'etale faithfully flat 
ring homomorphism $R\to R'$ such that $P\times_{W(R)}W(R')$ is trivial (i.e. has a section). 

b) The functor $P\mapsto FP$
 provides an equivalence from the category of 
$G$-torsors over $W(R)$ to the category of $L^+G$-torsors over $R$.
\end{proposition}
\begin{proof}
Let us first show part (a): Since $G$ is smooth, the base change $P_0=P\times_{W(R), w_0} R$ of $P$ by $w_0: W(R)\to R$ splits locally for the \'etale topology on $R$, i.e. there is
$R\to R'$ \'etale faithfully flat such that $P_0\times_RR'$ has a section $s_0'$. 
Now since $p$ is nilpotent in $R'$, the Witt ring $W(R')$ is separated and complete
for the topology defined by the powers of $I(R')={\rm ker}(W(R')\to R')$ (see \cite[Prop. 3]{Zinkdisplay}).
By Hensel's lemma, the section $s'_0$ lifts to a section $s'$ of $P\times_{W(R)}W(R')$.

Now let us prove part (b). We will first show that the functor is fully faithful. Let $\phi, \psi: P\to P'$ be (iso)morphisms of $G$-torsors
over $W(R)$ such that $F(\phi)=F(\psi)$. By part (a),  there is an \'etale faithfully flat $R\hookrightarrow R'$ such that $P\times_{W(R)} W(R')$, $P'\times_{W(R)}W(R')$, are both trivial $G$-torsors. Hence, $\phi$ and $\psi$ are given by multiplication by $g_\phi$, $g_{\psi}\in G(W(R'))$. Our assumption $F(\phi)=F(\psi)$ now quickly implies that $g_{\phi}=g_{\psi}$ and so $\phi=\psi$. It now remains to show the essential surjectivity of $P\mapsto FP$.
First notice that in the case $G=\GL_h$ this is provided by Zink's Witt descent (\cite[Prop. 33, Cor. 34]{Zinkdisplay}).
Indeed, a (fpqc locally trivial) $L^+\GL_h$-torsor over $R$ gives by definition ``Witt descent data'' on $W(R')^h$
with respect to the faithfully flat $R\to R'$; by  loc. cit.  these determine a projective finitely generated $W(R)$-module;
this is locally free of rank $h$ on $W(R)$ and its scheme of linear automorphisms  produce the desired $\GL_h$-torsor over $W(R)$. Let us now handle the case of a general reductive group $G$. There is a closed group scheme immersion
$i: G\hookrightarrow \GL_h$  and the fpqc quotient $\GL_h/G=\Spec(A)$ is represented by an affine scheme
(\cite[Prop. 6.11, Cor. 6.12]{ColliotSansuc}).
Suppose now that $Q$ is a $L^+G$-torsor over $R$ and consider the induced $L^+\GL_h$-torsor
\[
i(Q)=Q\times_{L^+G, i} L^+\GL_h.
\]
By the above discussion, there is a $\GL_h$-torsor $P$ over $W(R)$ such that $FP\cong i(Q)$;
this gives a closed immersion $j: Q\hookrightarrow FP$ of schemes over $R$ which is $L^+G$-equivariant. By descent, the quotient $P/G$ is represented by an affine 
$W(R)$-scheme $Z$ and $P\to Z=P/G$ is a $G$-torsor. Then by a similar argument as in the proof 
of Lemma \ref{torsor01}, $FP\to FZ$ is a $L^+G$-torsor. Now
by applying fpqc descent (i.e.  ``taking the quotient of $j$''
by the action of $L^+G$)  we obtain a morphism
\[
j/L^+G: \Spec(R)  \to  FZ
\]
which amounts to a $W(R)$-valued point of $Z$. The pull back of the $G$-torsor $P\to Z$ along this point gives the desired $G$-torsor over $W(R)$.\end{proof}

 \begin{corollary}\label{torsCor}
 Suppose  $S$ is a scheme in ${\rm Nilp}_{W(k_0)}$. Then all (fpqc locally trivial) $L^+G$-torsors over $S$ split locally for the \'etale topology of $S$.\qed
 \end{corollary}

\begin{lemma}
\label{torsor03}
a) Let $A_n$, $n\geq 1$, be a sequence of local $W(k_0)$-algebras. Then the 
functor from the category of $G$-torsors over $A=\prod_nA_n$ 
to the product of the categories of $G$-torsors over 
$A_n$ is an equivalence of categories.

b) Let $A$ be a $W(k_0)$-algebra with a descending chain of ideals $\fraka_n$, $n\geq 1$, with $\fraka_i\fraka_j\subset \fraka_{i+j}$ and $\fraka_1/\fraka_n$ nilpotent in $A_n:=A/\fraka_n$, for all $n$, and
such that $A\cong\varprojlim_n A_n$.
Then the
functor from the category of $G$-torsors over 
$A$ to the category of compatible systems of $G$-torsors over $A_n$  is an equivalence of categories.
\end{lemma}
\begin{proof}
In both cases (a) and (b), the full faithfulness is clear and follows from the fact that the scheme of isomorphisms between two $G$-torsors over $A$ is 
represented by an affine $A$-scheme. Also, when $G=\GL_h$, both functors are essentially surjective
(for (b) this follows from lifting of projective modules, see for example \cite[p. 146-148]{Zinkdisplay}).
 In general, pick a closed group scheme embedding $i:G\hookrightarrow\GL_h$.
 Consider a (compatible, for part (b)) sequence  of $G$-torsors $Q_n$ over $A_n$ and the corresponding 
 $\GL_h$-torsors $Q_n\times_{G}\GL_h$.
 (In case (a), since $A_n$ is local, 
 the $\GL_h$-torsor $Q_n\times_{G}\GL_h$ is trivial:   $\GL_h\times_AA_n\cong Q_n\times_{G}\GL_h$.)
 By essential surjectivity for $\GL_h$, there is a $\GL_h$-torsor $P$ over $A$  with (compatible)
 isomorphisms $P\times_AA_n\cong Q_n\times_{G}\GL_h$.
  As in the proof of Proposition \ref{torsor02}, the fpqc quotient $P/G$ is representable by an affine $A$-scheme $Z$
and $P\to Z$ is a $G$-torsor.
The $Q_n$'s give rise to a (compatible) sequence of elements of $Z(A_n)$, which yields a point $Z(A)$; this gives the desired $G$-torsor over $A$.
\end{proof}

\begin{remark}
\label{torsor04}
i) Lemma \ref{torsor03} remains true if $G$ is replaced by $L^+G$
provided that, in addition, all $A_n$ are in ${\rm ANilp}_{W(k_0)}$. Indeed, we see that the proof 
of part (a) goes through after
the observation 
that all $L^+\GL_h$-torsors over   $A_n$ are trivial, given by
free   $W(A_n)$-modules, of rank $h$.
This is a corollary of Zink's Witt descent (also obtained by combining Proposition \ref{torsor02} with 
\cite[Prop. 35]{Zinkdisplay}). For part (b), we can again apply Proposition \ref{torsor02} and observe 
that we have $W(A)\cong \varprojlim_n W(A_n)$ with $W(\fraka_1/\fraka_n)$ nilpotent in $W(A/\fraka_n)$.

ii) 
As a corollary of the above, it also follows that the equivalence of Proposition \ref{torsor02} (b) extends to the case that $R$ is a $p$-adically
complete and separated $W(k_0)$-algebra.
\end{remark}

\section{On  certain nilradicals}

Suppose that $R$ is a (commutative) $\f_p$-algebra. Denote by $F_R: R\to R$ the Frobenius
$F_R(x)=x^p$ so that $\ker(F_R^n)=\{x\in R\ |\ x^{p^n}=0\}$. 

\begin{definition} 
\label{GNN06}
\begin{itemize}
\item[a)] 
We will say that a $\f_p$-algebra $R$ 
is {\sl Frobenius separated  (Fs)} when,  for all $n\geq 1$, $R$ is 
$\ker(F_R^n)$-separated, i.e. when for all $n\geq 1$, we have
\begin{equation}
\label{GNN07}
\bigcap_m\ker(F^n_R)^m=0.
\end{equation}
\item[b)] 
We will say that a $\z_{(p)}$-algebra $R$ is 
{\sl almost Frobenius separated (aFs)} when there exists a nilpotent ideal 
$\fraka$ such that $p\in\fraka$ and $R/\fraka$ is Frobenius separated.
\end{itemize} 
\end{definition}

\begin{lemma}\label{propFs}
We have:
\begin{itemize}
\item[1)] Noetherian rings with $pR=0$ (resp. $p\in\sqrt{0_R}$) are (almost) Fs.  

\item[2)] If $R\subset R'$ and $R'$ is (almost) Fs, then so is $R$.

\item[3)] If $R_i$, $i\in I$, are Fs then  the product 
$\prod_{i\in I}R_i$ is Fs.
\end{itemize} 
\end{lemma}

\begin{proof}
These properties are easy consequences of the definition.
\end{proof}

The following is a slight generalization of \cite[lemma 5.1]{E7}:

\begin{lemma}
\label{GNN04}
Let $k$ be a field of characteristic $p$, and let $B$ 
and $B'$ be reduced $k$-algebras. The tensor product 
$R=B\otimes_kB'$ is Frobenius separated.
\end{lemma}
\begin{proof}
Notice that we can write $B\otimes_kB'\hookrightarrow \prod_i L_i\otimes_k \prod_j L'_j$ where 
$L_i$, $L'_j$ are field extensions of $k$; then $B\otimes_kB'\subset (\prod_i L_i)\otimes_k (\prod_j L'_j)\hookrightarrow 
\prod_{i, j} L_i\otimes_k L'_j$.
Hence,   as Frobenius separatedness is
inherited by products and subrings (Lemma \ref{propFs}), without loss of generality, we can 
assume that $B$ and $B'$ are algebraically closed fields.
We may clearly also assume $B=B'$. Zorn's lemma allows us to pick a maximal
separable subextension $k\subset S\subset B$, so that $B=S^{1/p^\infty}$. 
Notice that $T:=B\otimes_kS$ is a reduced ring. Let $\{x_i|i\in I\}$ be a $p$-basis 
of $S$ (i.e. a subset such that every element of $S$ has a unique representation as 
a sum $x=\sum_{\underline n}a_{\underline n}^p{\underline x}^{\underline n}$ where 
$\underline n=( n_i )_{i\in I}$ runs through the set of 
multiindices with $p-1\geq n_i\geq0$ and $n_i=0$ for almost all $i$). 
It is easy to see that $B$ is the quotient of the polynomial 
algebra $S[\{b_{i,e}|i\in I\,,e\geq  1\}]$ by the ideal which is generated by 
$b_{i,e+1}^p-b_{i,e}$ and $b_{i,1}^p-x_i$ for $i\in I$, $e\geq 1$. Now consider 
$$
r_{i,e}:=b_{i,e}\otimes1-1\otimes b_{i,e}\in R.
$$ 
It follows that $R$ is the quotient of the polynomial algebra 
$$
T[\{r_{i,e}|i\in I\,,e\geq 1\}]
$$ 
by the ideal which is generated by $r_{i,e+1}^p-r_{i,e}$ 
and $r_{i,1}^p$, for $i\in I$, $e\geq 1$.  We can deduce that, for each subset 
$I_0\subset I$, there exists a ring endomorphism
$\theta_{I_0}:R\rightarrow R$, defined by $\theta_{I_0}(r_{i,e})=r_{i,e}$ if $i\in I_0$,
$\theta_{I_0}(r_{i,e})=0$, if $i\notin I_0$.
Notice that each 
$\theta_{I_0}$ preserves $\fraka_n:=\ker(F_R^n)$, and that  
$\theta_{I_0}=\theta_{I_0}\circ\theta_{I_0}$.

Now we can see that $\fraka_n$ is generated by the elements $r_{i,n}$, so that 
$\fraka^\nu_n$ is generated by the set $\{\prod_ir_{i,n}^{\nu_i}|\sum_i\nu_i=\nu\}$. 
The only multiindices which give rise to non-zero products are bounded by 
$\nu_i\leq  p^n-1 $, and these are seen to involve factors indexed by at least 
$ \nu/{(p^n-1)}$ many elements of $I$. Consequently one has 
$\theta_{I_0}(\fraka^\nu_n)=0$ provided that $\#(I_0)<\nu/(p^n-1)$. 
Now consider some $x\in\bigcap_\nu\fraka^\nu_n$, and choose a large 
finite set $I_0$ with $\theta_{I_0}(x)=x$. We deduce 
$x\in\theta_{I_0}(\fraka^{1+(p^n-1)\#(I_0)})=0$.
\end{proof}

\begin{lemma}\label{FsBNI}
Suppose that $A$ is an almost Frobenius separated $\z_{(p)}$-algebra. Let 
$\fraka\subset A$ be an idempotent ideal, i.e. $\fraka=\fraka^2$, which is in addition bounded nilpotent, i.e. there exists $m\geq 1$ such that for all 
$x\in \fraka$, $x^m=0$. Then $\fraka$ is the zero ideal.
\end{lemma}

\begin{proof}
Consider the image $\bar\fraka$ of $\fraka$ in $ A/pA$. It is enough to show $\bar\fraka=0$; 
then $\fraka\subset pA$ and so $\fraka=\fraka^2\subset p^2A$  which gives $\fraka\subset p^mA=0$ for $m>>0$.
Therefore, it is enough to show the result when $pA=0$. Our assumption implies that there is $n\geq 1$ such that
$\fraka=\fraka^2\subset \ker(F^n_A)$. The result  follows since then $\fraka=\fraka^m\subset  \ker(F_A^n)^m$
which gives $\fraka\subset \cap_m (\ker(F_A^n))^m$. Therefore $\fraka$ is nilpotent and $\fraka=\fraka^2$ gives $\fraka=0$. 
\end{proof}

\begin{remark}
 In fact, the proof only uses that $\cap_m \ker(F_{A/pA})^m$ is nilpotent. Indeed, suppose that $pA=0$ and that
we have $\fraka=\fraka^2\subset \ker(F^n_A)$ with $n\geq 1$.  We can induct on $n$.   
Consider $\fraka'=\{\sum_ix_iy_i^p|x_i\in A,\,y_i\in\fraka\}$. This 
satisfies $\fraka'=\fraka'^2\subset   \ker(F_A^{n-1})$
and by induction $\fraka'=0$. This gives $\fraka\subset \ker(F_A)$ and we can conclude the proof.
\end{remark}

\begin{corollary}
\label{GNN05}
Let $R$ be a Noetherian ring.
If $p$ is nilpotent in $R$ and if $B$ is a flat $R$-algebra for which $\sqrt{0_B}$ is nilpotent, 
then $B\otimes_RB$ is almost Frobenius separated (aFs).
\end{corollary}
\begin{proof} Recall that we denote by $Q(A)$ the  total quotient ring  of $A$, i.e. $Q(A)$ is the localization
 $N^{-1}A$ on the set of non-zero divisors $N\subset A$.
Lemma \ref{GNN04} implies that $Q(B_{\rm red})\otimes_{Q(R_{\rm red})}Q(B_{\rm red})$ is almost Frobenius separated, 
and we claim that the kernel of the canonical homomorphism
$$
B\otimes_RB\rightarrow Q(B_{\rm red})\otimes_{Q(R_{\rm red})}Q(B_{\rm red})
$$
is nilpotent. We first argue that we can quickly reduce to the case that $R$ is reduced.
Indeed, since $R$ is Noetherian, $\sqrt{0_R}$ is a nilpotent ideal we can 
replace $R$ and $B$ by $R_{\rm red}$ and $B/\sqrt{0_R}B$. Now assuming that $R=R_{\rm red}$, let us 
write $S$ for the set of non-zero-divisors of $R$. Observe that since $B$ is $R$-flat
the images of the elements of $S$ in $B$ are also non-zero-divisors of $B$, yielding natural injections 
$B\hookrightarrow S^{-1}B$ and also $B\otimes_RB\hookrightarrow S^{-1}B\otimes_RS^{-1}B$.
However, since the kernels of both maps
$$
S^{-1}B\otimes_{Q(R)}Q(B_{\rm red})\rightarrow Q(B_{\rm red})\otimes_{Q(R)}Q(B_{\rm red}),
$$
and
$$
S^{-1}B\otimes_RS^{-1}B\rightarrow S^{-1}B\otimes_RQ(B_{\rm red})
$$
are nilpotent, the statement easily follows.
\end{proof}

In particular   Corollary \ref{GNN05} applies when 

\begin{itemize}
\item
$B=\prod_\frakm R_\frakm$, where $\frakm$ 
runs through the set of maximal ideals of the Noetherian ring $R$ and $R_\frakm$ stands for the localization at $\frakm$, and
\item
$R$ is a local Noetherian ring and $B=\hat R$ is its   
completion.
\end{itemize}

(In the first situation, the flatness of $B=\prod_\frakm R_\frakm$ over the Noetherian ring $R$ follows from \cite[Theorem 2.1]{Chase}.)

\end{appendix}

\providecommand{\bysame}{\leavevmode\hbox to3em{\hrulefill}\thinspace}
\providecommand{\MR}{\relax\ifhmode\unskip\space\fi MR }
\providecommand{\MRhref}[2]{%
  \href{http://www.ams.org/mathscinet-getitem?mr=#1}{#2}
}
\providecommand{\href}[2]{#2}

\end{document}